\documentclass[review,sort&compress,3p]{elsarticle}

 \usepackage{amsthm}
\usepackage{graphicx}
\usepackage{mathrsfs}
\usepackage{amsfonts}
\usepackage{amssymb}
\usepackage{color}
\usepackage{subfigure}
\usepackage{amsmath,latexsym}
\theoremstyle{plain}
\newtheorem{theorem}{Theorem}[section]
\newtheorem{lemma}{Lemma}[section]
\newtheorem{prop}{Proposition}[section]

\newcommand{\ignore}[1]{}

\newtheorem{remark}{Remark}[section]
\def\eps{\varepsilon}

\def\ds{\mathrm{d}\mathbf{s}}

\def\dx{\mathrm{d}\mathbf{x}}
\def\ds{\mathrm{d}s}

\begin{document}

\begin{frontmatter}



\title{An efficient threshold dynamics method for wetting on rough surfaces}

\author[a]{Xianmin Xu}
\ead{xmxu@lsec.cc.ac.cn}
\author[b]{Dong Wang}
\ead{dwangaf@connect.ust.hk}
\author[b]{Xiao-Ping Wang\corref{cor1}}
\ead{mawang@ust.hk}

\
\cortext[cor1]{Corresponding author}

\address[a]{LSEC, Institute of Computational
  Mathematics and Scientific/Engineering Computing,
  NCMIS, AMSS, Chinese Academy of Sciences, Beijing 100190, China}
  \address[b]{Department of Mathematics, The Hong Kong University of
                       Science and Technology, Hong Kong}

\begin{abstract}
The threshold dynamics method developed by Merriman, Bence and Osher (MBO)  is an efficient method for  simulating the motion by mean curvature flow when the interface is away from the solid boundary.  Direct generalization of  MBO-type methods to the wetting problem with interfaces  intersecting the solid boundary  is not easy because solving the heat equation in a general domain with a wetting boundary condition is not as efficient as it is with the original MBO method.  The dynamics of the contact point also follows a different  law compared with the dynamics of the interface away from the boundary.  In this paper,  we develop an efficient volume preserving threshold dynamics method for simulating wetting  on rough surfaces. This method is based on minimization of the weighted surface area functional over an extended domain that includes the solid phase.   The method is  simple, stable with $O(N \log N)$  complexity per time step and is not sensitive to the inhomogeneity or roughness  of the solid boundary.
\end{abstract}


\end{frontmatter}


\section{Introduction}

Wetting describes how a liquid drop  spreads on a solid surface.
The most important quantity in wetting is the contact angle
between the liquid surface and the solid surface \cite{Gennes03}.
When the solid surface is homogeneous, the contact angle for a static drop is given by the famous
Young's equation:
\begin{equation}\label{e:Young}
\cos\theta_Y=\frac{\gamma_{SV}-\gamma_{SL}}{\gamma_{LV}},
\end{equation}
where $\gamma_{SL}$, $\gamma_{SV}$ and $\gamma_{LV}$ are the solid-liquid, solid-vapor
and liquid-vapor surface energy densities, respectively. $\theta_Y$ is the so-called Young's angle \cite{Young1805}.
Mathematically,  Young's equation \eqref{e:Young} can be
derived by minimizing the total energy in the solid-liquid-vapor system. If we ignore gravity,
the total energy in the system can be written as
\begin{equation}\label{e:energy}
\mathcal{E}=\gamma_{LV}|\Sigma_{LV}|+\gamma_{SL}|\Sigma_{SL}|+\gamma_{SV}|\Sigma_{SV}|,
\end{equation}
where $\Sigma_{LV}$, $\Sigma_{SL}$ and $\Sigma_{SV}$ are respectively the liquid-vapor, solid-liquid and
solid-vapor interfaces, and $|\cdot|$ denotes the area of the interfaces.
When the solid surface $\Gamma$ is a  homogeneous planar surface,
under the condition that the volume of the drop
is fixed, the unique minimizer of the total  energy is a domain with a spherical surface in $\Omega$,
and the contact angle between the surface and the solid surface $\Gamma$ is Young's angle $\theta_Y$\cite{Whyman08}.

The study of wetting and contact angle hysteresis on rough surfaces is of critical importance for many applications and has
attracted much interest in the physics and applied mathematics communities
\cite{Quere08,Extrand02,Alberti05,XuWang2013,erbil2014}.
Numerical simulation of  wetting on rough surfaces is  challenging.  One must track the interface motion accurately, as well as  deal with complicated boundary shapes and boundary conditions.  There are many different types of numerical methods for solving interface and contact line problems,  including the
front-tracking method \cite{womble1989front,leung2009grid}, the front-capturing method using the level-set function \cite{zhao1996variational}, the phase-field methods \cite{elliott2003computations,bertozzi2007inpainting}, among others \cite{deckelnick2005computation}.

 Merriman, Bence and Osher (MBO) developed an efficient threshold dynamics method  to simulate the motion by mean curvature flow \cite{merriman1992diffusion,merriman1994motion}. This method is based on the observation
that the level-set of the solution of a  heat equation moves in normal direction
 at a velocity equal to the mean curvature of the level-set surface.   The method  alternately diffuses and sharpens characteristic
functions of regions and is easy to implement and highly efficient. The method was also  extended to problems with volume preservation \cite{ruuth2003simple, kublik2011algorithms}
and to some high-order geometric flow problems \cite{esedoglu2008threshold}.
Recently, Esedoglu and Otto  extended the threshold dynamics method to the multi-phase problems with arbitrary surface tension \cite{esedoglu2013threshold}. There have been many studies on the MBO threshold dynamics method, including some efficient implementations
\cite{ruuth1998efficient,ruuth1998diffusion,Svadlenka2014} and convergence analysis \cite{barles1995simple,evans1993convergence,chambolle2006convergence,ishii2005optimal}.
 In particular, Laux and collaborators  established the convergence
of some computational algorithms including one with volume preservation\cite{Laux16a,Laux16b}.

The generalization of MBO-type methods to the wetting problem where interfaces intersecting the boundary is not straightforward  because of  a  lack of integral representation with a heat kernel for a general domain.
In the original MBO scheme, when the interface does not intersect the solid boundary,  one can solve the heat equation efficiently on a rectangular domain with a uniform grid using convolution of the heat kernel with the initial condition \cite{ruuth1998efficient,ruuth1998diffusion}.  The convolution can be evaluated using fast Fourier transform (FFT)  at $N\log N$ cost per time step where $N$ is the total number of grid points.
One way to generalize  MBO-type methods to wetting on solid surfaces  is to solve the heat equation  with a wetting boundary condition before the  volume-preserving thresholding step.  However,  the usual fast algorithms cannot be applied for this case,  especially when the boundary is rough.

In this paper, we aim to develop an efficient  volume-preserving threshold dynamics method
for solving wetting problems on rough surfaces.
Our method is based on the approach of  Esedoglu-Otto  \cite{esedoglu2013threshold}.
 The key idea  is
to extend the original domain with a  rough boundary to a regular cube and treat the solid part as another phase.  In the thresholding step, the solid phase domain remains unchanged.
We show that the algorithm has the total interface energy decaying property and our numerical results show that the equilibrium interface satisfies  Young's equation near the contact point. The advantage of the method is that it can be implemented efficiently on uniform meshes with a fast algorithm (e.g. FFT) since the computational domain is rectangular and we can simulate wetting on rough boundaries of any shape. We also introduce a fast algorithm for volume preservation based on a quick-sort algorithm and a time refinement scheme to improve the accuracy of the solution at
the contact line.



The  paper proceeds as follows. In Section 2, we introduce the surface energies of the wetting problem. A direct (but less efficient) MBO-type threshold dynamics
method for solving wetting problems is also described.  In Section 3, we introduce a  new threshold dynamics method which is simple, efficient and  easy to implement. Several  modifications of the method are also discussed.  { In Section 4,
we discuss the implementation of the algorithm and perform the accuracy check.  We also introduce a quick-sort algorithm  for volume preservation and a time refinement technique to improve the accuracy of the contact point motion.   In Section 5 and Section 6, we present  numerical examples of wetting on rough surfaces to   demonstrate  the efficiency of the new method.}

\section{The minimization of  surface energies }
%
%
\begin{figure}
\begin{center}
\includegraphics[scale=0.45]{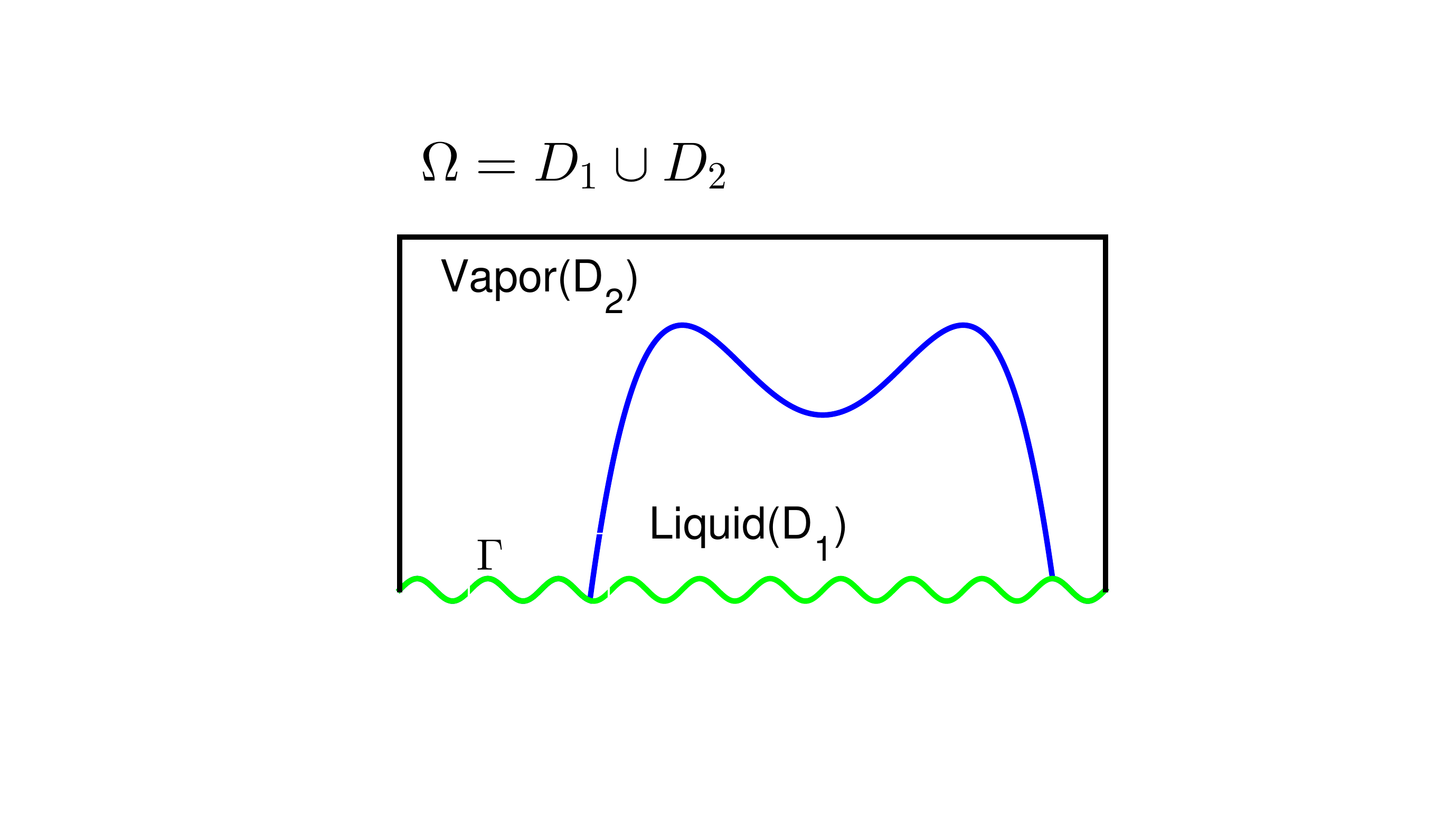}
\end{center}
\vspace{-1.5cm}
\caption{Wetting on a rough surface}
\label{fig:2phase1}
\end{figure}
We consider a wetting problem in a domain $\Omega\in \mathbb{R}^n$, $n=2,3$ (see Figure~\ref{fig:2phase1}).
The solid surface $\Gamma$ is part of the domain boundary $\partial\Omega$.
Denote the liquid domain by  $D_1\subset\Omega$. For simplicity, we assume that
$\partial D_1\cap \partial\Omega\subset\Gamma$.  The volume of the liquid drop is fixed such that
$|D_1|=V_0$.  We denote $\Sigma_{LV}=\partial D_1\cap \Omega$,
$\Sigma_{SL}=\partial D_1\cap\Gamma$ and $\Sigma_{SV}=\Gamma\setminus\partial D_1$  as the liquid-vapor,
 solid-liquid and  solid-vapor interfaces respectively.
Then, the equilibrium configuration of the system is obtained by  minimizing the total interface energy of the system as follows:
 \begin{align}\label{e:minE}
 \min_{|D_1|=V_0} \mathcal{E}(D_1)=\gamma_{LV}|\partial D_1\cap\Omega|
 +\int_{\partial D_1\cap\Gamma}\gamma_{SL}(x)\ds+\int_{\Gamma\setminus\partial D_1}\gamma_{SV}(x)\ds
 \end{align}
where the solid  boundary $\Gamma$ is rough and/or  chemically inhomogeneous (i.e.
$\gamma_{SL}(x)$ and $\gamma_{SV}(x)$ may depend on $x$). {To ensure  the problem is
well-posed, Young's angle must satisfy $0<\theta_Y<\pi$. By equation~\eqref{e:Young}, this leads
to the condition $-1<\frac{\gamma_{SV}-\gamma_{SL}}{\gamma_{LV}}<1$. Throughout the paper,
we will assume this condition holds.}


To solve  problem \eqref{e:minE} numerically,
it is convenient to use a diffuse interface model
to approximate the sharp interface energy. Suppose
$\varphi$ is a phase-field function, such that $D_1=\{\varphi<0\}$ represents
the liquid domain, $\{\varphi>0\}$ represents
the vapor domain and $\Sigma_{LV}=\{\varphi=0\}$ is the liquid-vapor interface.
 The total energy
\eqref{e:energy} can be  approximated by
\begin{equation}\label{e:ph_energy}
\mathcal{E}_\eps^{ph}(\varphi)=\int_{\Omega}\eps|\nabla\varphi|^2 +\frac{f(\varphi)}{\eps}\dx+\int_{\Gamma}\gamma(x,\varphi)\ds,
\end{equation}
where $\eps$ is a small parameter representing interface thickness, $f(\varphi)=\frac{(1-\varphi^2)^2}{4}$ is a double-well function and $$\gamma(\varphi)=\frac{\tilde{\gamma}_{SV}(x)+\tilde \gamma_{SL}(x)}{2}+\frac{\tilde \gamma_{SV}(x)-\tilde  \gamma_{SL}(x)}{4}(3\varphi-\varphi^3).$$
It can be proved that when $\eps$ goes to zero, after scaling, the energy in \eqref{e:ph_energy} converges to that in \eqref{e:energy} \cite{XuWang2011}. Therefore, problem \eqref{e:minE} can be approximated by minimizing
the total energy $\mathcal{E}_\eps^{ph}$ under the volume constraint $\int_{\Omega}(\varphi-1)/2\dx=V_0$.

The $H^{-1}$ gradient flow of the energy functional \eqref{e:ph_energy} will lead to a Cahn-Hilliard equation with contact angle boundary conditions \cite{chenWangXu2014}.  Alternatively, the $L^2$ gradient flow will lead to a modified Allen-Cahn equation:
\begin{align}\label{e:AC}
\left\{
\begin{array}{lr}
\varphi_t=\eps\Delta\varphi-\frac{f'(\varphi)}{\eps}+\delta & \hbox{ in }\Omega; \medskip\\
\frac{\partial\varphi}{\partial n}+{\gamma'(x,\varphi)}=0, & \hbox{ on } \Gamma; \medskip\\
\frac{\partial\varphi}{\partial n}=0, & \hbox{ on } \partial\Omega\setminus\Gamma,\medskip\\
\int_{\Omega}\frac{\varphi-1}{2}\dx=V_0.
\end{array}
\right.
\end{align}
Here $\delta$ is a Lagrangian multiplier for the volume constraint.

A MBO-type  threshold dynamics scheme can be derived easily using a splitting method for  \eqref{e:AC}.
Assume  we have a solution $\varphi^k$ (characteristic function of a region)  at the {$k$-th} time step. We can  first solve the heat equation
\begin{align}
\label{e:heat1}
\left\{\begin{array}{lr}
\bar{\varphi}_t= \eps\Delta\bar \varphi & \hbox{ in }\Omega.\medskip\\
\frac{\partial \bar\varphi}{\partial n}+\gamma'(x,\bar \varphi)=0, & \hbox{ on } \Gamma,\medskip\\
\frac{\partial\bar\varphi}{\partial n}=0,& \hbox{ on } \partial\Omega\setminus\Gamma,\medskip\\
\bar\varphi(x,0)=\bar\varphi^{k},
\end{array}
\right.
\end{align}
for some time $\delta t_1$ and then solve
\begin{align}\label{e:split2}
\left\{
\begin{array}{l}
\varphi_t=-\frac{f'(\varphi)}{\eps} \medskip \\
\varphi(x,0)=\bar{\varphi}(x,\delta t_1)
\end{array}
\right.
\end{align}
for some time $\delta t_2$ and set $\varphi^{k+1}=\varphi(x,\delta t_2)$.
It is easy to see that  when $\delta t_2/\eps$ is large enough, solving the second equation \eqref{e:split2} is
reduced to a thresholding step
\begin{equation}\label{e:threshold}
\varphi(x,\delta t_2)\approx
\left\{
\begin{array}{lr}
-1 & \hbox{ if } \varphi(x,0)<0;\\
1& \hbox{ if } \varphi(x,0)>0,
\end{array}
\right.
\end{equation}
which gives a characteristic function $\varphi^{k+1}$ at the $k+1$ time step.
 This leads to the following  MBO-type scheme for the wetting problem:

\vspace{0.5cm}
{{\bf A direct MBO threshold dynamics scheme for the wetting problem}
\it
\begin{itemize}
\item[]
\begin{description}
\item[Step 0.] Given an initial domain $D_1^0\subset\Omega$ such that  $|D_1^{0}|=V_0$.
Set a tolerance parameter $\eps>0$.
\item[Step 1.] For any $k$,  we first solve the heat equation
\begin{align}
\label{e:heat2}
\left\{\begin{array}{lr}
{\varphi}_t= \Delta  \varphi & \hbox{ in }\Omega, \medskip\\
\frac{\partial \phi}{\partial n}+\gamma'(x, \varphi)=0, & \hbox{ on } \Gamma, \medskip \\
\frac{\partial\varphi}{\partial n}=0, & \hbox{ on } \partial\Omega\setminus\Gamma,\medskip\\
\vspace{1mm}
\varphi(x,0)=\chi_{D_1^k},
\end{array}
\right.
\end{align}
for a  time step $\delta t$.
\item[Step 2.] Determine a new $D_1^{k+1}$ using thresholding
\begin{equation*}
D_1^{n+1}=\{x:\varphi(x,\delta t)<\frac12+\delta\}.
\end{equation*}
Here $\delta$ is chosen such that the volume $|D_1^{k+1}|=V_0$.
\item[Step 3.] If $|D_1^{k}-D_1^{k+1}|<\eps$, stop; otherwise, set $k=k+1$ and
go back to Step 1.
\end{description}
\end{itemize}
}
\vspace{0.5cm}
In the original MBO scheme, when the interface does not intersect the solid boundary,  one can solve the heat equation efficiently on a uniform grid using convolution of the heat kernel with the initial condition \cite{ruuth1998efficient,ruuth1998diffusion}.  The convolution can be evaluated using FFT  at $M\log(M)$ cost per time step where $M$ is the total number of grid points.   However, when the interface intersects the solid boundary, one must solve the heat equation with the wetting boundary condition as in \eqref{e:heat2}.  In this case, and in particular for  rough boundaries, the usual fast algorithms cannot be applied to solve \eqref{e:heat2}.  In the next section, we will introduce a new threshold dynamics method.


\section{A new threshold dynamics method for the wetting problem}
In this section, we  introduce a new threshold dynamics method  motivated by the recent work of Esedoglu and Otto \cite{esedoglu2013threshold}.
The main idea is to extend the fluid domain $\Omega$ to a larger domain containing
the solid phase. In the extended domain,  the interface energies between different phases in \eqref{e:minE}
can be approximated by a convolution of characteristic functions and a Guassian kernel (see details below). We then derive a simple scheme to  minimize the new energy functional with
the constraint that the solid phase does not change and the volume of
the liquid phase is preserved. The scheme leads to a new threshold dynamics method for solving the  wetting problem.

\subsection{The representation of interface energies in an extended domain}
In the following, we let $D_1,D_2\subset\Omega$ be  the liquid and  vapor phases, respectively.
Let $\Sigma_{LV}=\partial D_1\cap\partial D_2$ be the liquid-vapor interface.
When $\delta t \ll 1$, the area of $\Sigma_{LV}$ can be approximated by (see \cite{alberti1998non,miranda2007short})
\begin{equation}\label{e:3.1}
|\Sigma_{LV}|\approx \frac{1}{\sqrt{\delta t}}\int \chi_{D_1} G_{\delta t}*\chi_{D_2} \dx,
\end{equation}
where $\chi_{D_i}$ is the characteristic function of $D_i$ and
$$G_{\delta t}(\mathbf x)=\frac{1}{(4\pi \delta t)^{n/2}}\exp(-\frac{|\mathbf x|^2}{4\delta t})$$
is the Gaussian kernel.
\begin{figure}
\begin{center}
\includegraphics[scale=0.45]{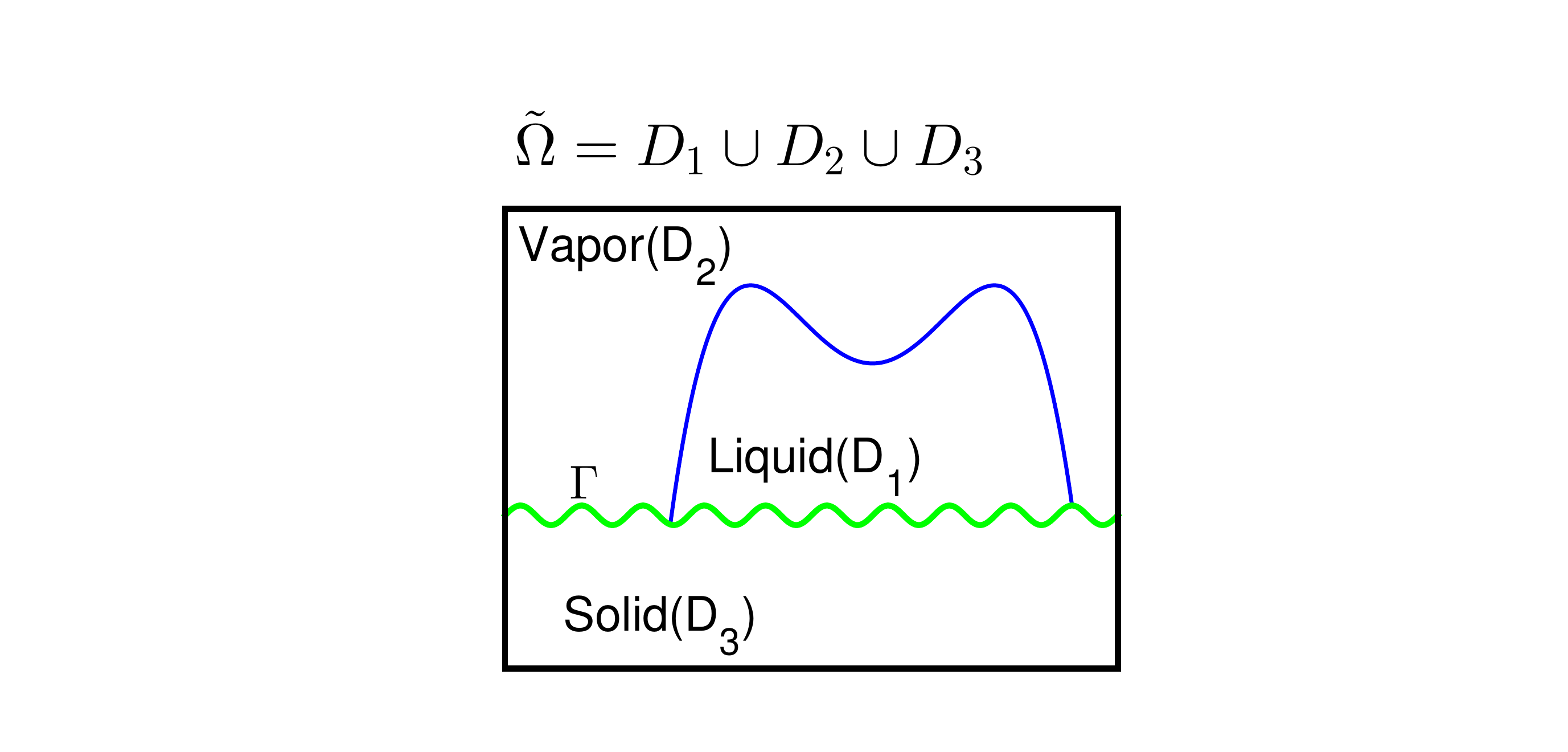}
\end{center}
\caption{ Extended computational domain  $\tilde{\Omega}=\Omega \cup D_3$}
\label{fig:phasefigure}
\end{figure}

In the total energy \eqref{e:minE}, the second and third terms are surface energies defined on the solid surface $\Gamma$. They are the solid-liquid interfacial energy term on $\Sigma_{SL}=\partial D_1\cap\Gamma$ and the solid-vapor
interfacial energy term on $\Sigma_{SV}=\partial D_2\cap\Gamma$. To  approximate the two terms using the Gaussian kernel,
 we extend the domain $\Omega$  beyond $\Gamma$ (see Figure~\ref{fig:phasefigure}).
The extended domain  is $\tilde{\Omega}=\Omega \cup D_3$ where $D_3$ is the solid region.
Then, the solid surface is $\Gamma=\partial \Omega\cap \partial D_3$,  the solid-liquid interface is
 $\Sigma_{SL}=\partial D_1\cap\partial D_3$
and the solid-vapor interface is $\Sigma_{SV}=\partial D_2\cap\partial D_3$.
Similar  to \eqref{e:3.1}, the total energy $\mathcal{E}$ in \eqref{e:minE} can be approximated by
\begin{align}\label{e:ApproEng}
&\mathcal{E}^{\delta t}(\chi_{D_1},\chi_{D_2})=\nonumber \\
&\frac{\gamma_{LV}}{\sqrt{\delta t}}\!\int_{\tilde{\Omega}}\!\chi_{D_1}G_{\delta t}*\chi_{D_2} \dx
+\frac{\gamma_{SL}}{\sqrt{\delta t}}\!\int_{\tilde{\Omega}}\!\chi_{D_1}G_{\delta t}*\chi_{D_3} \dx+
\frac{\gamma_{SV}}{\sqrt{\delta t}}\!\int_{\tilde{\Omega}}\!\chi_{D_2}G_{\delta t}*\chi_{D_3}\dx.
\end{align}
For simplicity, we assume $\gamma_{SL}$ and $\gamma_{SV}$ are constants throughout
this section. The  analysis and the algorithms can be easily generalized to  cases
where they are not homogeneous.  In  section 5, we will apply the method
to a chemically patterned surface where $\gamma_{SL}$ and $\gamma_{SV}$ are
piecewise constant functions.

Denote $u_1=\chi_{D_1}$ and $u_2=\chi_{D_2}$. We define an admissible set
\begin{align}\label{e:admissible}
\mathcal{B}=\{(u_1,u_2)\in BV(\Omega)\ |\ &u_i(x)=0, 1, \hbox { and } u_1(x)+u_2(x)=1,\ a.e.\ x\in\Omega, \nonumber \\
 &\int_{\Omega} u_1\dx=V_0  \}
\end{align}
The wetting problem \eqref{e:minE} can be approximated by
\begin{equation}\label{e:minEt}
\min_{(u_1,u_2)\in\mathcal{B}} \mathcal{E}^{\delta t}(u_1,u_2).
\end{equation}
This is a nonconvex minimization problem since $\mathcal{B}$ is not a convex set.
The $\Gamma$-convergence of  problem \eqref{e:minEt} to \eqref{e:minE} can be proved in a similar way as in \cite{esedoglu2013threshold}.
\subsection{Derivation of the threshold dynamics method}\label{sec:deriva}
We will derive the threshold dynamics method for  problem \eqref{e:minEt}. {Notice that the  problem is to minimize a concave energy functional defined on a nonconvex admissible set. We first show that it can be
relaxed to a problem defined on a convex admissible set}. Then we derive a threshold dynamics method for the equivalent problem.  The relaxed problem is given by
\begin{equation}\label{e:relaxedPb}
\min_{(u_1,u_2)\in\mathcal{K}} \mathcal{E}^{\delta t}(u_1,u_2).
\end{equation}
where $\mathcal{K}$
is  the convex hull of the admissible set $\mathcal{B}$:
\begin{equation}\label{e:convexhull}
\mathcal{K}=\{(u_1,u_2)\in BV(\Omega) | 0\leq u_i\leq 1, u_1(x)+u_2(x)=1,\ a.e.\ x\in\Omega, \int_{\Omega} u_1\dx =V_0\}.
\end{equation}
The following lemma shows that the relaxed problem \eqref{e:relaxedPb} is equivalent to the original problem \eqref{e:minEt}.
For convenience later, we prove the result for a  slightly more general problem with an extra linear functional term $\mathcal{L}(u_1, u_2)$.
\begin{lemma}\label{lem:relax}
For any given $\alpha,\beta \geq 0$ and any linear functional $\mathcal{L}(u_1, u_2)$, we have
$$
\min_{(u_1,u_2)\in\mathcal{K}} (\alpha \mathcal{E}^{\delta t}(u_1, u_2) +\beta \mathcal{L}(u_1,u_2))=\min_{(u_1,u_2)\in\mathcal{B}} (\alpha \mathcal{E}^{\delta t}(u_1, u_2) +\beta \mathcal{L}(u_1,u_2)).$$
\end{lemma}
\begin{proof}Let  $(\tilde{u}_1,\tilde{u}_2)\in\mathcal{K}$ be a minimizer of the functional
$$\alpha \mathcal{E}^{\delta t}(u_1, u_2) +\beta \mathcal{L}(u_1,u_2).$$
Since $\mathcal{B}\subset\mathcal{K}$, we have
\begin{align*}
\alpha \mathcal{E}^{\delta t}(\tilde u_1, \tilde u_2) +\beta \mathcal{L}(\tilde u_1,\tilde u_2)&=\min_{(u_1,u_2)\in\mathcal{K}}\alpha \mathcal{E}^{\delta t}(u_1, u_2) +\beta \mathcal{L}(u_1,u_2)\\
&\leq \min_{(u_1,u_2)\in\mathcal{B}}\alpha \mathcal{E}^{\delta t}(u_1, u_2) +\beta \mathcal{L}(u_1,u_2).
\end{align*}
Therefore, we need only to prove  that $(\tilde{u}_1,\tilde{u}_2)\in\mathcal{B}$.

 The proof  is trivial when $\alpha=0$, since
the minimizer of a linear functional in a convex set must belong to the boundary of the set.  When $\alpha>0$,
 we prove  by contradiction.  If $(\tilde{u}_1,\tilde{u}_2)\not\in\mathcal{B}$, there is a set $A\in\Omega$
 and a constant $0<C_0<\frac12$, such that $|A|>0$ and
$$
0<C_0<\tilde u_1(x),\tilde u_2(x)<1-C_0, \ \ \ \hbox{for all } x\in A.
$$
We divide $A$ into two sets $A=A_1\cup A_2$ such that $A_1 \cap A_2=\emptyset$ and $|A_1|=|A_2|=|A|/2$.
Denote $u_1^t=\tilde{u}_1+t \chi_{A_1}-t\chi_{A_2}$ and
$u_2^t=\tilde{u}_2-t \chi_{A_1}+t\chi_{A_2}$.
When $0<t<C_0$, we have $0<u_1^t,u_2^t<1$ and
\begin{equation*}
u_1^t+u_2^t= \tilde{u}_1+\tilde{u}_2=1,\ \hbox{and } \int_{\Omega} u_1^t\dx=\int_{\Omega}\tilde{u}_1\dx=V_0.
\end{equation*}
This implies that $(u_1^t,u_2^t)\in \mathcal{K}$.
Furthermore, direct computations give,
\begin{align*}
\frac{d^2}{dt^2}(\alpha\mathcal{E}^{\delta t}(u_1^t,u_2^t)+\beta \mathcal{L}(u_1^t,u_2^t))
&=\frac{1}{\sqrt{\delta t}}\int_{\tilde{\Omega}}\frac{d}{dt} u_1^t G_{\delta t}*\frac{d}{dt} u^t_2\dx\\
&=\frac{1}{\sqrt{\delta t}}\int_{\tilde{\Omega}}(\chi_{A_1}-\chi_{A_2}) G_{\delta t}*(\chi_{A_2}-\chi_{A_1})\dx\\
&=-\frac{1}{\sqrt{\delta t}}\int_{\tilde{\Omega}}(\chi_{A_1}-\chi_{A_2}) G_{\delta t}*(\chi_{A_1}-\chi_{A_2})\dx\\
&<0.
\end{align*}
The functional is concave on the point $(\tilde{u}_1,\tilde{u}_2)$.
Thus, $(\tilde{u}_1,\tilde{u}_2)$ cannot be a minimizer  of the functional. This contradicts the assumption.
\end{proof}

The above lemma implies that we can solve the relaxed problem \eqref{e:relaxedPb} instead of the original one \eqref{e:minEt}.
In the following, we show that the problem can be solved iteratively using a thresholding method.

Suppose we solve problem \eqref{e:relaxedPb} using an iterative method. In the $k^{th}$ step,
we have an approximated solution $(u_1^k,u_2^k)$. The energy functional  $\mathcal{E}^{\delta t}(u_1,u_2)$ can be linearized near
the point $(u_1^k,u_2^k)$ as follows:
\begin{align*}
\mathcal{E}^{\delta t}(u_1,u_2)\approx \mathcal{E}^{\delta t}(u_1^k,u_2^k)+\hat {\mathcal{L}}(u_1-u_1^k,u_2-u_2^k,u_1^k,u_2^k)+ h.o.t.
\end{align*}
with
\begin{align}
\hat  {\mathcal{L}}&(u_1,u_2, u_1^k,u_2^k)=\nonumber\\
&\frac{1}{\sqrt{\delta t}}\left(\int_{\tilde{\Omega}}u_1G_{\delta t}*(\gamma_{LV}u_2^k+\gamma_{SL}\chi_{D_3})\dx
+
\int_{\tilde{\Omega}}u_2G_{\delta t}*(\gamma_{LV}u_1^k+\gamma_{SV}\chi_{D_3})\dx\right).
\label{e:linearization}
\end{align}
Then we minimize the linearized functional
\begin{equation}\label{e:minL}
\min_{(u_1,u_2)\in\mathcal{K}} \hat{\mathcal{L}}(u_1,u_2,u_1^k,u_2^k)
\end{equation}
 and set the solution to  $(u_1^{k+1},u_2^{k+1})$. By Lemma~\ref{lem:relax},
 the solution to \eqref{e:minL} is in $\mathcal{B}$. In other words, $u_1^{k+1}$ and $u_2^{k+1}$
 are characteristic functions of some proper sets $D_1^{k+1}$ and $D_2^{k+1}$ such that $|D_1^{k+1}|=V_0$.

The following lemma shows that the minimizing problem \eqref{e:minL} is solved via
a simple thresholding approach.
\begin{lemma}\label{lem:linearization}
Denote
\begin{equation}\label{e:phi12}
\phi_1=\frac{1}{\sqrt{\delta t}} G_{\delta t}*(\gamma_{LV} u_2^k+\gamma_{SL}\chi_{D_3}),\
\phi_2=\frac{1}{\sqrt{\delta t}} G_{\delta t}*(\gamma_{LV} u_1^k+\gamma_{SV}\chi_{D_3}).
\end{equation}
Let
\begin{equation}\label{e:D1}
D_1^{k+1}=\{ x\in\Omega | \;\phi_1 < \phi_2+\delta \}
\end{equation}
for some $\delta$ such that $|D_1^{k+1}|=V_0$.  Define $D_2^{k+1}=\Omega\setminus D_1^{k+1}$.
Then $(u_1^{k+1},u_2^{k+1})=(\chi_{D_1^{k+1}},\chi_{D_2^{k+1}})$ is a solution to \eqref{e:minL}.
\end{lemma}
\begin{proof}
By Lemma~\ref{lem:relax}, we need only to prove
\begin{align}
\hat{\mathcal{L}}(u_1^{k+1},u_2^{k+1},u_1^k,u_2^k)\leq \hat{\mathcal{L}}(u_1,u_2,u_1^k,u_2^k),
\end{align}
for all $(u_1,u_2)\in\mathcal{B}$.

For each $(u_1,u_2)\in\mathcal{B}$, we know $u_1=\chi_{\hat{D}_1}$ and $u_2=\chi_{\hat{D}_2}$ for some open sets $\hat{D}_1$,
$\hat{D}_2$ in $\Omega$, such that $\hat D_1\cap \hat D_2=\emptyset$, $\hat D_1\cup \hat D_2=\Omega$ and $|\hat D_1|=V_0$.
Let $A_1=\hat D_1\setminus  D_1^{k+1}=D_2^{k+1}\setminus \hat D_2$ and $A_2=\hat D_2 \setminus D_2^{k+1}=D_1^{k+1}\setminus \hat D_1$.
We must have $|A_1|=|A_2|$ due to the volume conservation property. Since $A_1\subset D_2^{k+1}$, we have
$$\phi_1(x)\geq \phi_2(x)+\delta,\ \ \ \ \forall x\in A_1.$$
Similarly, since $A_2\in D_1^{k+1}$, we have
$$\phi_1(x)< \phi_2(x)+\delta,\ \ \ \ \forall x\in A_2.$$
Therefore, we have
\begin{align*}
&\hat{\mathcal{L}}(u_1^{k+1},u_2^{k+1},u_1^k,u_2^k)-\hat{\mathcal{L}}(u_1,u_2,u_1^k,u_2^k)\nonumber\\
=&\int_{\tilde{\Omega}} (u_1^{k+1}-u_1)\phi_1+(u_2^{k+1}-u_2)\phi_2\dx \\
=&-\int_{A_1}\phi_1\dx+\int_{A_2}\phi_1\dx -\int_{A_2}\phi_2\dx+\int_{A_1}\phi_2\dx\\
=&\int_{A_1}(\phi_2-\phi_1)\dx+\int_{A_2}(\phi_1-\phi_2)\dx\\
\leq &-\delta \int_{A_1}\dx+\delta \int_{A_2}\dx =0.
\end{align*}

\end{proof}

We are led to the following threshold dynamics algorithm:

\vspace{0.5cm}

{{\bf Algorithm I: }
\it
\begin{itemize}\item[]
\begin{description}
\item[Step 0.] Given initial $D_1^0, D_2^0\subset\Omega$, such that $D_1^0\cap D_2^0=\emptyset$, $D_1^0\cup D_2^0=\Omega$ and
$|D_1^0|=V_0$. Set a tolerance parameter $\eps>0$.
\item[Step 1.]For given sets $(D_1^k, D_2^k)$, we define two functions
\begin{equation}\label{e:newphi12}
\phi_1=\frac{1}{\sqrt{\delta t}} G_{\delta t}*(\gamma_{LV} \chi_{D_2^k}+\gamma_{SL}\chi_{D_3}),\
\phi_2=\frac{1}{\sqrt{\delta t}} G_{\delta t}*(\gamma_{LV} \chi_{D_1^k}+\gamma_{SV}\chi_{D_3}).
\end{equation}
\item[Step 2.] Find a $\delta$ such that the set
\begin{equation}\label{e:threshold1}
\tilde{D}_1^{\delta}=\{ x\in\Omega | \phi_1 < \phi_2+\delta. \}
\end{equation}
satisfies $|\tilde{D}_1^{\delta}|=V_0$.
Denote $D_1^{k+1}=\tilde{D}_1^{\delta}$ and $D_2^{k+1}=\Omega\setminus D_1^{k+1}$.
\item[Step 3.] If $|D_1^{k}-D_1^{k+1}|\leq \eps$, stop; otherwise,
go back  to Step 1.
\end{description}
\end{itemize}
}
\vspace{0.5cm}

\begin{remark}
 The method is  simple and easy to implement.

(1) We can always extend $\Omega$ to a cubic domain $\tilde{\Omega}$, since
the only constraints on the extension are $D_1\in\tilde{\Omega}$ and $|D_1|=V_0$. For
the cube domain, the convolution in \eqref{e:newphi12} can be computed by fast
algorithms (e.g. the FFT).

(2) To keep the volume of the
liquid phase unchanged, we need to find a proper $\delta$ in Step 2.
This can be done by using an iterative method (such as bisection method), as shown in \cite{ruuth2003simple} for mean curvature flow. In the next section, we will give a simpler and more efficient technique to
determine $\delta$.

(3) The above derivation of the thresholding method for the wetting problem can
be easily generalized to a multiphase system with wetting boundary conditions, e.g. the three-phase system \cite{shi2014modeling}, in the same spirit of  Esedoglu and Otto \cite{esedoglu2013threshold}.
\end{remark}

\subsection{A simplified algorithm  for the two-phase problem}
For the two-phase problem, Algorithm I can be simplified as follows.
Noticing that $u_1+u_2=1$ in $\Omega$, we actually have only one unknown $u_1$
in \eqref{e:relaxedPb}. Define
$$
\mathcal{K}_1=\{u\in BV(\Omega) | 0\leq u\leq 1,\ a.e.\ x\in\Omega, \int_{\Omega} u \dx =V_0\}.
$$
 It is easy to see that \eqref{e:relaxedPb} can be rewritten as
\begin{align}
\min_{u_1\in\mathcal{K}_1}\tilde{\mathcal{E}}^{\delta t}(u_1)=&-\gamma_{LV}\int_{\tilde{\Omega}} u_1 G_{\delta t}* u_1\dx
+\gamma_{LV}\int_{\tilde\Omega}u_1G_{\delta}*\chi_{\Omega}\dx\nonumber\\
&+\int_{\tilde \Omega} (\gamma_{SL}-\gamma_{SV}) u_1 G_{\delta t}*\chi_{D_3} \dx
+\int_{\tilde{\Omega}}\gamma_{SV}\chi_{\Omega}G_{\delta t}*\chi_{D_3}\dx.
\end{align}
Suppose we solve the problem using an iterative method.
For any given $u_1^k$, we could linearize the functional as
\begin{align*}
\tilde{\mathcal{E}}^{\delta t}(u_1)=\tilde{\mathcal{E}}^{\delta t}(u_1^k)+\tilde{\mathcal{L}}(u-u_1^k,u_1^k) + h.o.t.
\end{align*}
with
\begin{align}
\tilde{\mathcal{L}}(u,u_1^k)=&-2 \gamma_{LV}\int_{\tilde \Omega} u_1 G_{\delta t}*u_1^k \dx
+\gamma_{LV}\int_{\tilde\Omega}u_1G_{\delta}*\chi_{\Omega}\dx\nonumber\\
&+\int_{\tilde \Omega}(\gamma_{SL}-\gamma_{SV}) u_1 G_{\delta t}*\chi_{D_3} \dx\nonumber\\
=& \gamma_{LV}\int_{\tilde \Omega} u_1 G_{\delta t}*(u_2^k-u_1^k)\dx+\int_{\tilde \Omega}(\gamma_{SL}-\gamma_{SV}) u_1 G_{\delta t}*\chi_{D_3} \dx \nonumber\\
=&\gamma_{LV}\int_{\tilde \Omega} u_1G_{\delta t}*(u_2^k-u_1^k-\cos\theta_Y\chi_{D_3})\dx. \label{e:linear_New}
\end{align}
Here we use Young's equation $\gamma_{LV}\cos\theta_Y=\gamma_{SV}-\gamma_{SL}$.

As in the previous subsection, for the linearized functional \eqref{e:linear_New}, we can prove the  following result. The proof is  similar to that  for Lemma~\eqref{lem:linearization}.
\begin{lemma}\label{lem:linearization1}
Suppose $u_1^k=\chi_{D_1^{k}}$ for some sets $D_1^k\subset\Omega$ and $D_2^{k}=\Omega\setminus D_1^{k}$. Denote
\begin{equation*}\label{e:phi_New0}
\phi=\frac{\gamma_{LV}}{\sqrt{\delta t}} G_{\delta t}*( \chi_{D_2^k} -\chi_{D_1^k} -\cos(\theta_Y)\chi_{D_3}),
\end{equation*}
Let
$
\tilde{D}_1^{\delta}=\{ x\in\Omega\ |\ \phi< \delta \},
$
with some $\delta$ such that $|D_1^{k+1}|=V_0$.
Then $u_1^{k+1}=\chi_{D_1^{k+1}}$ is a minimizer of $\tilde{\mathcal{L}}(u,u_1^k)$ in $\mathcal{K}_1$.
\end{lemma}

This leads to
the following  algorithm.
\vspace{0.5cm}

{ {\bf   Algorithm II:}
\it
\begin{itemize}\item[]
\begin{description}
\item[Step 0.] Given initial $D_1^0\subset\Omega$, such that $|D_1^0|=V_0$. Set a tolerance parameter $\eps>0$.
\item[Step 1.]For given set $D_1^k$, set $D_2^k=\Omega\setminus D_1^k$, define a function
\begin{equation}\label{e:phi_New}
\phi=\frac{\gamma_{LV}}{\sqrt{\delta t}} G_{\delta t}*( \chi_{D_2^k} -\chi_{D_1^k} -\cos(\theta_Y)\chi_{D_3}).
\end{equation}
\item[Step 2.] Find a  $\delta\in(-1,1)$, so that the set
\begin{equation}\label{e:threshold2}
\tilde{D}_1^{\delta}=\{ x\in\Omega\ |\ \phi< \delta. \}
\end{equation}
satisfying $|\tilde{D}_1^{\delta}|=V_0$.
Denote $D_1^{k+1}=\tilde{D}_1^{\delta}$.
\item[Step 3.] If $|D_1^{k}-D_1^{k+1}|\leq \eps$, stop;  otherwise,
go back  to Step 1.
\end{description}
\end{itemize}
}
\vspace{0.5cm}

The following proposition shows that Algorithm I and Algorithm II are equivalent.
\begin{prop}\label{prop:3.1}
For any domain $(D_1^k, D_2^k)\in \mathcal{B}$, after one iteration,
\textit{Algorithm I} and \textit{Algorithm II} generate the same $(D_1^{k+1}, D_2^{k+1})$.
\end{prop}
\begin{proof}
We need only consider the thresholding equations \eqref{e:threshold1} and \eqref{e:threshold2}.
Direct computations give
\begin{align*}
\phi_1-\phi_2=&\frac{1}{\sqrt{\delta t}} G_{\delta t}*(\gamma_{LV} \chi_{D_2^k}+\gamma_{SL}\chi_{D_3})-\frac{1}{\sqrt{\delta t}} G_{\delta t}*(\gamma_{LV} \chi_{D_1^k}+\gamma_{SV}\chi_{D_3})\nonumber\\
=&\frac{1}{\sqrt{\delta t}} G_{\delta t}*\Big(\gamma_{LV} (\chi_{D_2^k}-\chi_{D_1^k})+(\gamma_{SL}-\gamma_{SV})\chi_{D_3}\Big)\\
=&
\frac{\gamma_{LV}}{\sqrt{\delta t}}
 G_{\delta t}*(\chi_{D_2^k}-\chi_{D_1^k}-\cos\theta_Y\chi_{D_3})=\phi.
\end{align*}
In the last equation, we used Young's equation.
Therefore, the thresholding equation \eqref{e:threshold1} is equivalent to the thresholding equation  \eqref{e:threshold2}.

\end{proof}

\subsection{Stability analysis}
In this subsection, we will show that the two algorithms above are stable,
in the sense that the total energy of $\mathcal{E}^{\delta t}$ always decreases in
the algorithm for any $\delta t>0$. We have the following theorem.

\begin{theorem}
Denote $(u_1^k,u_2^k)=(\chi_{D_1^k},\chi_{D_2^k}), \; k=0,1,2,...$, obtained in Algorithm I (or Algorithm II).
We have
\begin{equation}\label{e:energyDecay}
\mathcal{E}^{\delta t}(u_1^{k+1},u_2^{k+1})\leq \mathcal{E}^{\delta t}(u_1^{k},u_2^{k}),
\end{equation}
for all $\delta t>0$.
\end{theorem}
\begin{proof}By Proposition~\ref{prop:3.1}, we need only to prove  the theorem for Algorithm I.
By the definition of the linearization $\hat{\mathcal{L}}$ and  Lemma~\ref{lem:linearization}, we know that
\begin{align*}
&\mathcal{E}^{\delta t}( u_1^k,u_2^k)
+\frac{\gamma_{LV} }{\sqrt{\delta t}}
\int_{\tilde{\Omega}}u_1^k G_{\delta t}* u_2^k\dx
=\hat{\mathcal{L}}(u_1^k,u_2^k,u_1^k,u_2^k)
\\
&\geq \mathcal{L}(u_1^{k+1},u_2^{k+1},u_1^k,u_2^k)=\mathcal{E}^{\delta t}(u_1^{k+1},u_2^{k+1})\\
&\quad +\frac{\gamma_{LV} }{\sqrt{\delta t}}\left(\int_{\tilde{\Omega}} u_1^{k+1}G_{\delta t}* u_2^k\dx
+\int_{\tilde{\Omega}} u_2^{k+1}G_{\delta t}*u_1^k\dx
-\int_{\tilde{\Omega}} u_1^{k+1}G_{\delta t}* u_2^{k+1}\dx
\right).
\end{align*}
This leads to
\begin{align}\label{e:temp}
\mathcal{E}^{\delta t}(u_1^{k},u_2^k)\geq
\mathcal{E}^{\delta t}(u_1^{k+1},u_2^{k+1})+I,
\end{align}
with
{\begin{align*}
I=\frac{\gamma_{LV} }{\sqrt{\delta t}}\Big( &
\int_{\tilde{\Omega}} u_1^{k+1}G_{\delta t}* u_2^k\dx
+\int_{\tilde{\Omega}} u_2^{k+1}G_{\delta t}*u_1^k\dx
\\
&
-\int_{\tilde{\Omega}} u_1^{k+1}G_{\delta t}* u_2^{k+1}\dx
-\int_{\tilde{\Omega}} u_1^{k}G_{\delta t}* u_2^{k}\dx
\Big)\\
=-\frac{\gamma_{LV} }{\sqrt{\delta t}}&\int_{\tilde{\Omega}}(u_1^{k+1}-u_1^{k})G_{\delta t}* (u_2^{k+1}-u_2^k)\dx.
\end{align*}
By the fact that $u_1^{k}+u_2^{k}=u_1^{k+1}+u_2^{k+1}$,
 we have
\begin{align*}
I
&=\frac{\gamma_{LV} }{\sqrt{\delta t}}
\int_{\tilde{\Omega}}(u_1^{k+1}-u_1^k)G_{\delta t}*(u_1^{k+1}-u_1^k)\dx\geq 0.
\end{align*}}
This inequality together with \eqref{e:temp} implies \eqref{e:energyDecay}.
\end{proof}


\section{Numerical implementation and accuracy check}
In this section, we will introduce several techniques used to implement the algorithm efficiently.




\subsection{Calculation of convolution}

In Algorithm I, we need to calculate the  two convolutions $G_{\delta t}*(\gamma_{LV} \chi_{D_2^k}+\gamma_{SL}\chi_{D_3})$ and $G_{\delta t}*(\gamma_{LV} \chi_{D_1^k}+\gamma_{SV}\chi_{D_3})$ in an extended domain $\tilde{\Omega}$ which we can always choose to be a rectangular domain.  We can  use FFT to efficiently calculate the convolutions when the functions are periodic.
\ignore{
As discussed in the previous section, we can suppose the extended domain $\tilde{\Omega}$
is a cube. For such a domain, we can partition it uniformly with a mesh size $\delta x$.
 Furthermore, we can assume periodic condition on $\tilde{\Omega}$, since
the boundary condition does not change the movement of the interface
$\partial D_1\cap\partial D_2$ and $\partial D_1\cap\partial D_3$. (Here we suppose the width
of $D_3$ is of the same order of $\Omega$ that is much larger than $\delta t$.)
From the fundamental knowledge in Fourier analysis, we know:
\begin{align}
\mathcal{F}(f*g)=\mathcal{F}(f)\mathcal{F}(g)
\end{align}
where $\mathcal{F}$ represent the  Fourier transform operator. Discretely, we apply fast Fourier transform(FFT) on it so that the discrete convolution can be calculated by:
\begin{align}
f*g=\mathbf{iFFT}(\mathbf{FFT}(f).\times\mathbf{FFT}(g))
\end{align}
where $.\times$ is defined as $h(i,j)=f(i,j)\times g(i,j) \,  \forall i,j=1,2,\cdots,n$ if $h=f.\times g$ when $f$ and $g$ are two $n\times n$ matrix.
It is well-known that the computational complexity of FFT is $O(NlogN)$ with $N$ representing the number of grid points.
In summary, in step 1 in Algorithm 2, the computational complexity is $4NlogN+4	N+O(1)$ if the discrete Fourier transform of $G_{\delta t}$ is precomputed.
}
In our simulation, the characteristic functions (e.g. $\gamma_{LV} \chi_{D_2^k}+\gamma_{SL}\chi_{D_3}$)  are not periodic.
To calculate convolutions for non-periodic functions,  we can  further extend the domain by reflection  so that the functions are periodic in the extended domain.   However, the heat kernel $G_{\delta t}$ decays exponentially  and is  negligible  when $|x|>10\sqrt{\delta t}$. When we calculate the convolution, each target point will only be affected by a few neighboring points. Hence, if we apply the FFT without extending the computational domain, we will only have some error near the boundary of the computational domain (See Figure \ref{fig:ExtendvsNoextend}).  When the dynamic interface is far away from the boundary of the computational domain,  the solutions calculated with or without  the domain  extension  are the same, after the thresholding step.
Therefore, in our calculation, we always directly apply the FFT without extending the computational domain.

\begin{figure}
\begin{center}
\includegraphics[scale=0.15]{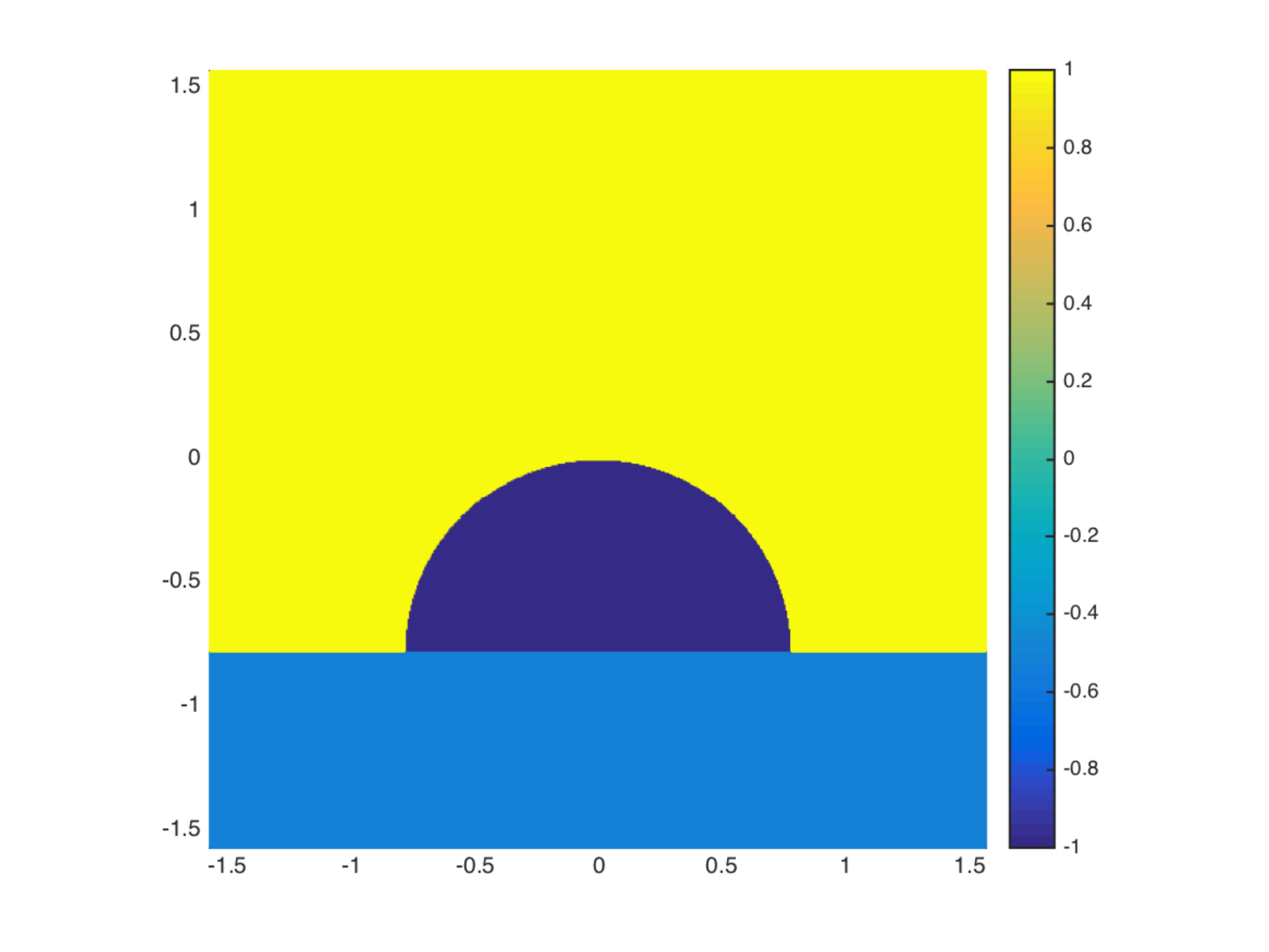} \includegraphics[scale=0.15]{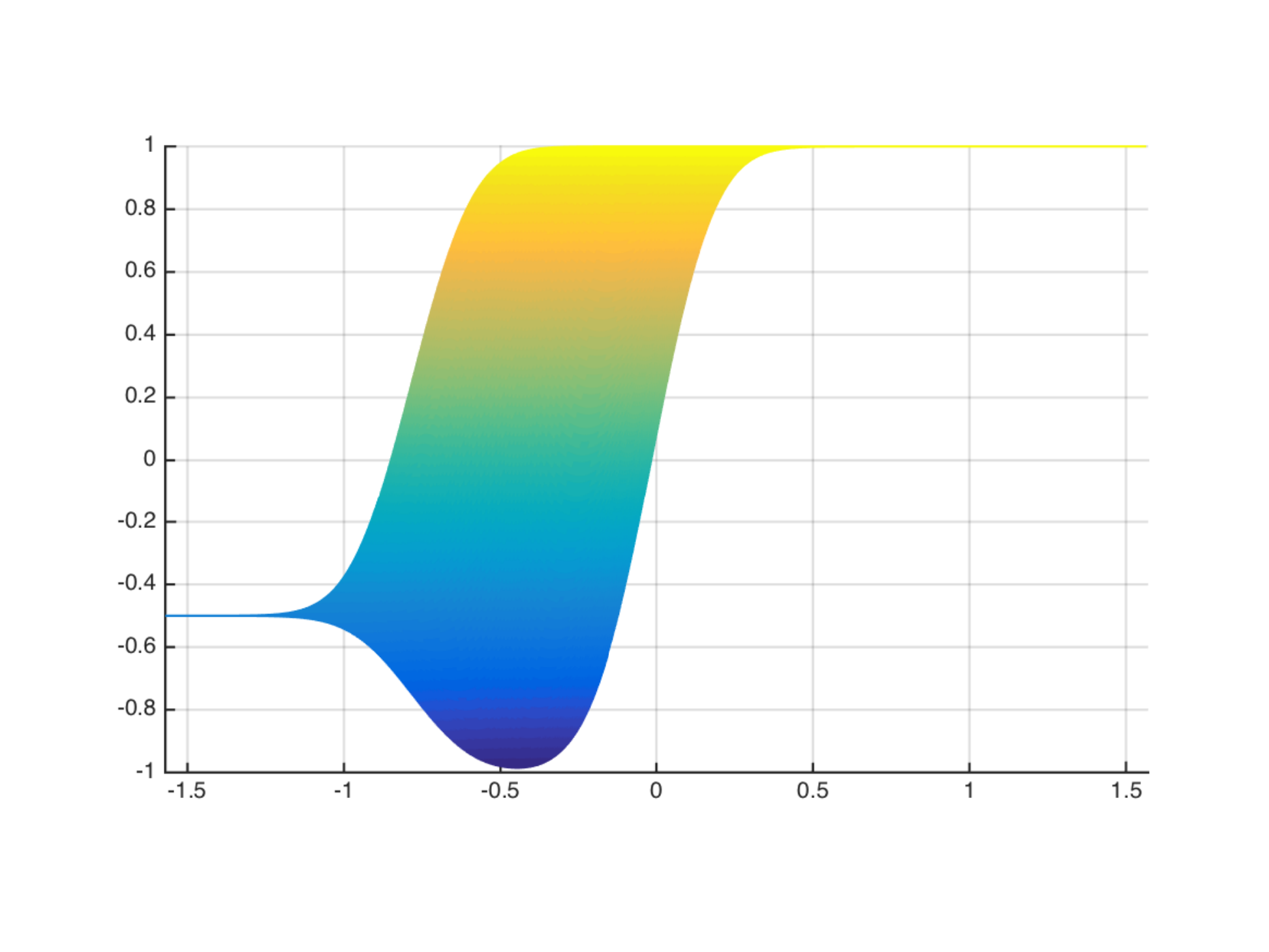} \includegraphics[scale=0.15]{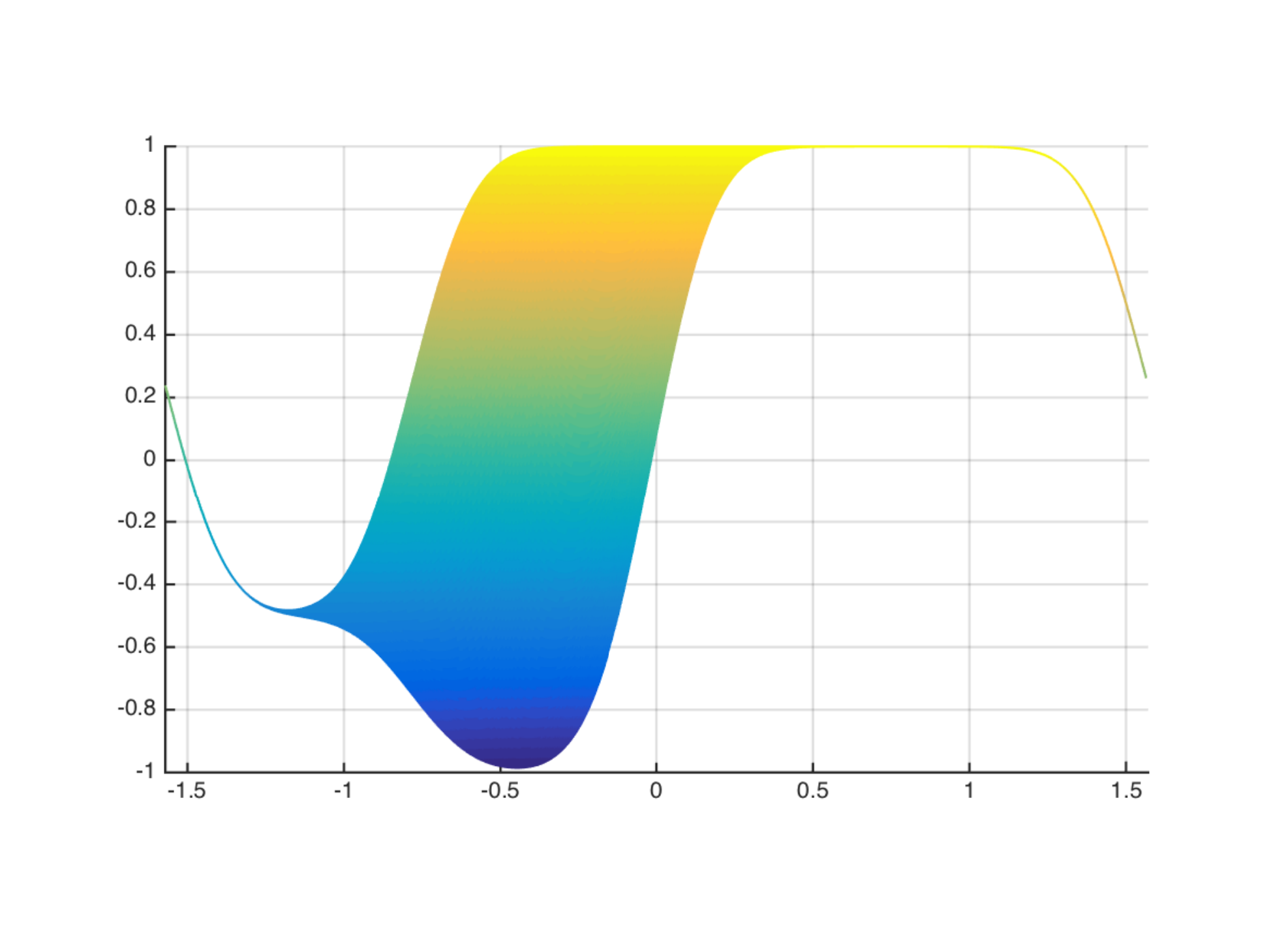} \includegraphics[scale=0.15]{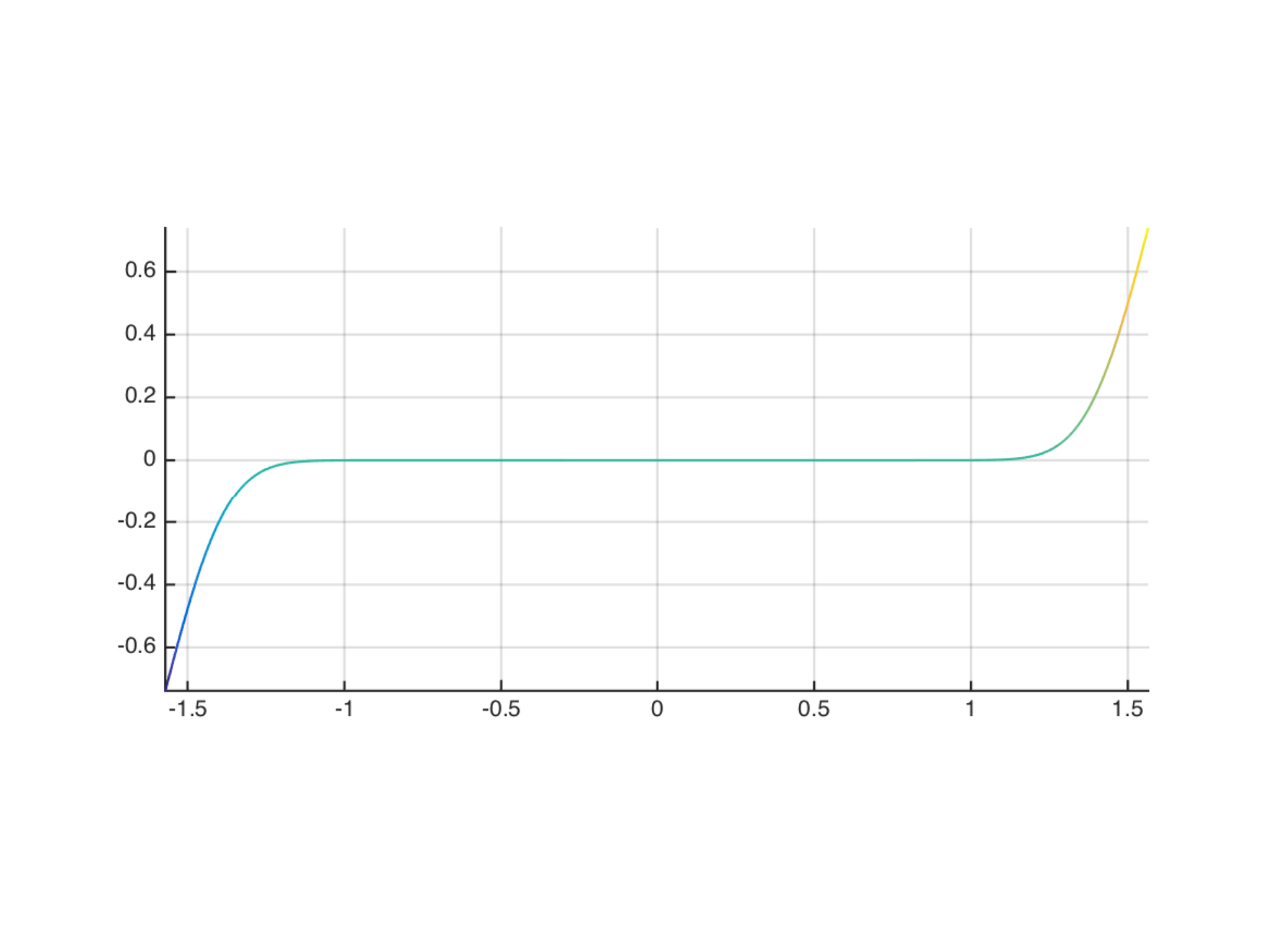}
\end{center}
\caption{From left to right: 1. The initial condition defined by the characteristic function of the domains. 2. The X-Z plane of the convolution between the heat kernel and the initial condition calculated by extending the domain by reflection (i.e. extending $[-\frac{\pi}{2},\frac{\pi}{2}]\times[-\frac{\pi}{2},\frac{\pi}{2}]$ to  $[-\frac{\pi}{2},\frac{3\pi}{2}]\times[-\frac{\pi}{2},\frac{3\pi}{2}]$). 3. The X-Z plane of the  convolution between the heat kernel and the initial condition calculated without extending the domain. 4. The difference between the second figure and the third figure.} \label{fig:ExtendvsNoextend}
\end{figure}


\subsection{A quicksort algorithm for  volume preservation} \label{sec:volumeconservation}
 In Step 2 of Algorithm I,  we need to enforce  volume preservation.  This is achieved by shifting the thresholding level by  $\delta$ as in (\ref{e:threshold1}).   The usual way to find $\delta$  is by some iteration method (e.g. bisection method, Newton method,  fixed point iteration, see \cite{ruuth2003simple}).   However, these iterative methods may be sensitive to the initial guess. In this section, we will introduce a direct and more efficient algorithm  to find a proper $\delta$.  If we consider a uniform mesh (in two dimensions) and denote the mesh size by $dx$,  the volume of a domain can be approximated by $V_0  \approx M\times dx^2$ (with first-order accuracy).  To maintain the same volume after thresholding, what we need to do in Step 2  is to find a threshold $\delta $ such that there are $M$ grid values of $\phi_1-\phi_2=g$ which are  less than $\delta$. Since we have the values of $\phi_1-\phi_2$ at each grid point, we  can use the quicksort algorithm (available in Matlab) \cite{hoare1962quicksort} to sort the values in ascending order into a list $\mathcal{S}=\{g_1, g_2, ...,g_M, g_{M+1}, ...\}$. We then take the average of the $M^{th}$ value and $(M+1)^{th}$ value in the ordered list $\mathcal{S}$, i.e. $\delta=\frac{g_{M}+g_{M+1}}{2}$ to be the threshold value $\delta$.  {A simple example to demonstrate this fast scheme is shown in Figure~\ref{fig:quicksortex}}. The scheme is summarized as  follows:
 {
 \begin{figure}
 \centering
 \includegraphics[scale=0.42]{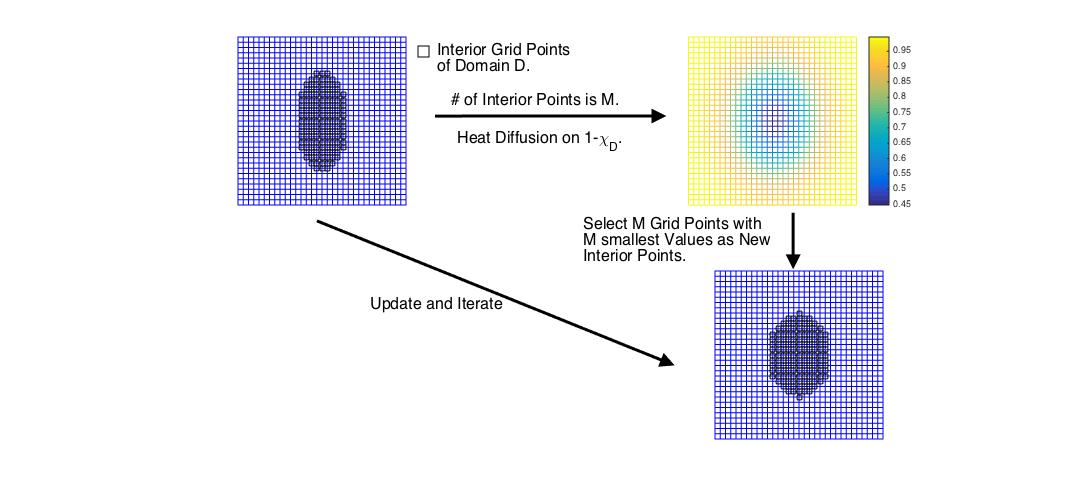}
 \caption{An example to demonstrate our new scheme for volume preservation. Initially, there are $M$ interior points. After convolution, we select M grid points with the $M$ lowest values as the new interior points.} \label{fig:quicksortex}
 \end{figure}
}
\bigskip

 {\bf{A quicksort scheme for volume conservation}
\it
\begin{itemize}\item[]
\begin{description}
\item[Step 0.] Set $V_0$ as the volume to be preserved and $M$ as the integer part of $V_0/dx^2.$
\item[Step 1.] Use a quicksort algorithm to sort $g=\phi_1-\phi_2$, which is defined in Step 2 in Algorithm 2, in ascending order into a list $\mathcal{S}=\{g_1, g_2, ...,g_M, g_{M+1}, ...\}$.
\item[Step 2.] Set  $\delta=\frac{g_{M}+g_{M+1}}{2}$.
\end{description}
\end{itemize}
}

\bigskip
In summary,  the computational complexity involved in  finding $\delta$ is $O(N\log(N))$ when the quicksort algorithm is used.
It is straightforward to see that this scheme will give the same $\delta$ as the iterative scheme proposed by Ruuth \cite{ruuth2003simple} (with first order accuracy).  However, our scheme costs much less computationally.

\bigskip
\subsection{Accuracy check of  Algorithm I}
{In this subsection, we will
check the accuracy of Algorithm I.
We first consider an example of  motion of two circles. One circle is centered at $(0.35,0.35)$ with radius  $0.2$ and the other is centered at $(0.7,0.7)$ with radius  $0.15$ (see  in Figure \ref{fig:QuickSortSmoothInitial}).
The volume-preserving
mean curvature flow is governed by the interface motion law $v_n = \kappa -\kappa_a$, where $v_n$ represents the normal velocity of the interface, $\kappa$ is the curvature and $\kappa_a$ is the average  curvature of the interface. By this motion, the larger circle will grow in volume while the smaller circle will shrink gradually.  The exact solution can be calculated  and the area enclosed by the smaller circle after a time $t=0.02$ is given by $0.0445079$ \cite{ruuth2003simple}. Using Algorithm I, we compute numerically the motion of the two circles  and compare the results with the exact solution.  Table~\ref{tab:AccuracyCheckQuickSort} shows the volume error as well as  $L^{\infty}$ error compared with the exact solution  at $t = 0.02$  for different $\delta t = (0.002,0.001,0.0005,0.00025$) with the same spatial resolution ($4096\times 4096$). The results  indicate the first-order accuracy of our scheme.
\begin{figure}[h]
\centering
\includegraphics[scale=0.36]{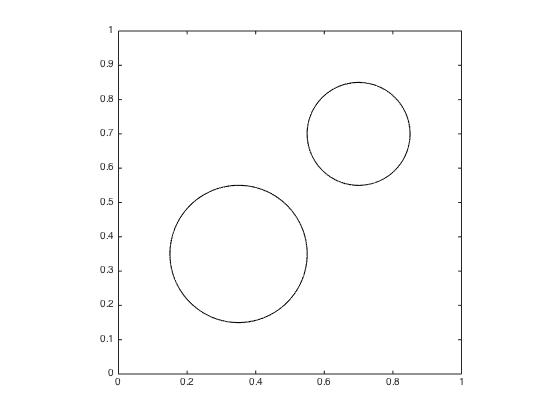}
\includegraphics[scale=0.36]{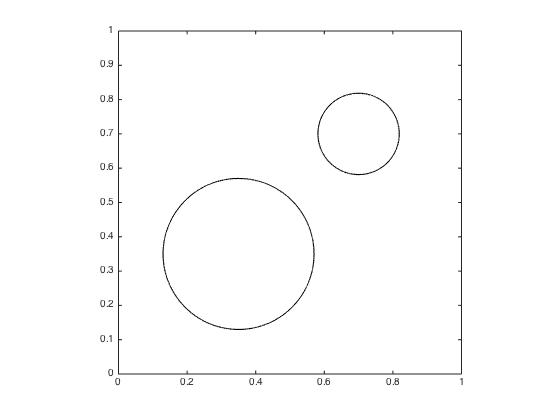}
\caption{The two circles at $t=0$ (left) and  $t=0.02$(right).}\label{fig:QuickSortSmoothInitial}
\end{figure}
\begin{table}
{
\caption{Accuracy check of Algorithm I for the two circle motion }\label{tab:AccuracyCheckQuickSort}
\begin{center}
\begin{tabular}{|c|c|c|c|c|}
  \hline
$ \delta t$ &  Volume Error  &Convergence Rate & $ L^{\infty}$ Error & Convergence Rate\\
\hline
0.002 &-0.0033&--&0.0074& --\\
\hline
 0.001 &-0.0015&1.15&0.0023&2.20\\
\hline
  0.0005&-0.00079&1.14&0.0011&1.08\\
\hline
  0.00025&-0.00035&1.03&0.00061&0.81\\
 \hline
\end{tabular}
\end{center}
}
\end{table}
}

{
We next  consider the  motion of two semi-circles on the solid surface. One is centered at $(0.3,0.25)$ with radius  $0.2$ and  the other one is centered at $(0.8,0.25)$ with radius  $0.15$ (See Figure~\ref{fig:QuickSortNonsmoothInitial}). In this problem, we set  Young's angle to  $\pi/2$. Then the wetting boundary condition in \eqref{e:AC}
will reduce to a homogeneous Neumann boundary condition. One can obtain the same motion as
that of two full circles by  symmetric reflection. Therefore,  the exact solution can also be calculated and
 the area enclosed by the smaller semi-circle after a time $t=0.02$ is $0.022254$ (half of the volume of the smaller circle in the previous example). Again,  using Algorithm I, we compute numerically the motion of the two circles and compare the results with the exact solution.  Table~\ref{tab:AccuracyCheckQuickSortNonsmooth} shows the volume error as well as  $L^{\infty}$ error compared with the exact solution  at $t = 0.02$  for different $\delta t = (0.002,0.001,0.0005,0.00025$) with the same spatial resolution ($4096\times 4096$).
The results  show first-order accuracy for volume preservation but a half-order convergence for $L^{\infty}$ norm.  This is typical
for multi-phase problems with a junction.
\begin{figure}[h]
\centering
\includegraphics[scale=0.36]{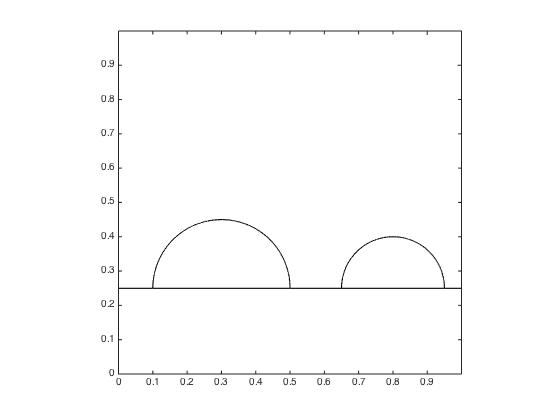}
\includegraphics[scale=0.36]{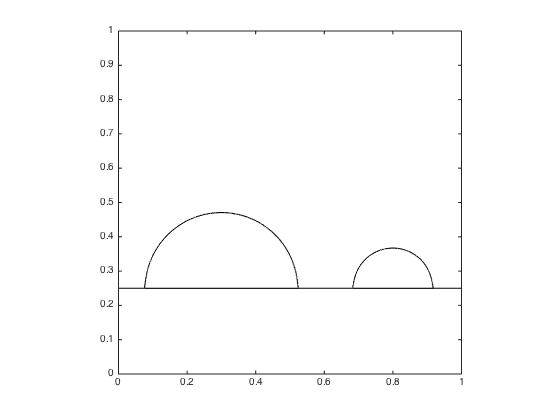}
\caption{The two semi-circles at $t=0$ (left) and  $t=0.02$(right).}\label{fig:QuickSortNonsmoothInitial}
\end{figure}
\begin{table}
{
\caption{Accuracy check of Algorithm I for the motion of two semi-circles on solid boundary.}\label{tab:AccuracyCheckQuickSortNonsmooth}
\begin{center}
\begin{tabular}{|c|c|c|c|c|}
  \hline
$ \delta t$ &  Volume Error  &Convergence Rate & $ L^{\infty}$ Error & Convergence Rate\\
\hline
0.002 &-0.0071&-- & 0.030& --\\
\hline
 0.001 &-0.0033&1.17 &0.011&1.78\\
\hline
  0.0005&-0.0016&1.03&0.0057&0.88\\
\hline
  0.00025&-0.00092&0.75&0.0035&0.62\\
 \hline
\end{tabular}
\end{center}
}
\end{table}
}

{
 Finally, we consider the case of drop spreading on a solid surface with a general static contact angle.  The initial drop is a semi-circle centered at $(0,-\frac{\pi}{4})$ with a radius $\frac{\pi}{4}$ (see Figure~\ref{fig:AccuracyCheckTime}).  We  set three surface tensions to  $\gamma_{LV}=1,\gamma_{LS}=1$ and $\gamma_{SV}=1+\sqrt{3}/2$ which gives  Young's  angle $\frac{\pi}{3}$.
Thus the drop will spread  and  its contact angle decreases gradually to  Young's angle.
Since we do not know the exact solution in this case, the reference solution of the liquid-vapor interface after time $t=0.2$ is numerically computed by choosing  sufficiently  small $\delta t = 0.00125$ and $dx =\frac{\pi}{4096}$. We then compare the numerical solution with the reference solution for different  $\delta t = 0.04,0.02, 0.01, 0.005$ but  with the same $dx=\frac{\pi}{4096}$. The results are shown in Table~\ref{tab:timeconvergencerate}, which suggests a first-order
convergence rate in $L^1$ error and a half-order convergence rate in $L^{\infty}$ error.
\begin{figure}[h]
\centering
\includegraphics[scale=0.3]{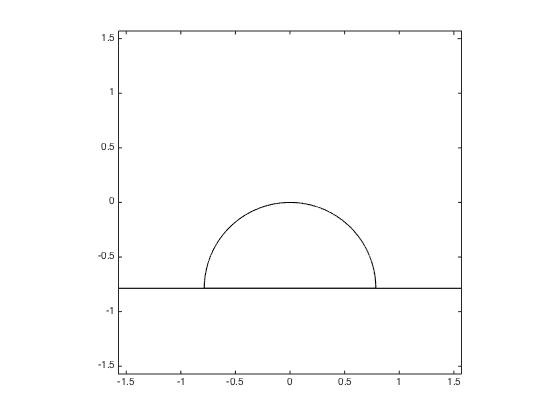}
\includegraphics[scale=0.3]{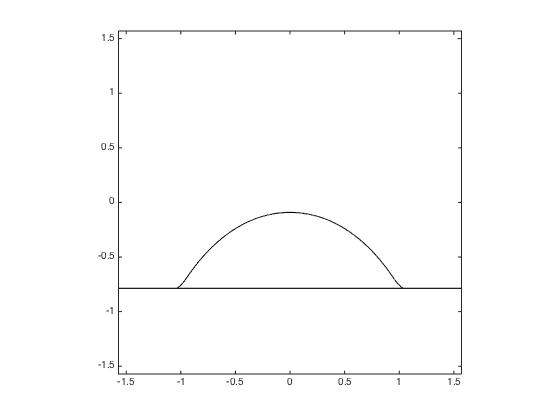}
\caption{The shape of a drop at $t=0$ (left) and  $t=0.2$(right).}\label{fig:AccuracyCheckTime}
\end{figure}
\begin{table}
{
\caption{Accuracy check w.r.t the time step $\delta t$ for the drop spreading problem. }\label{tab:timeconvergencerate}
\begin{center}
\begin{tabular}{|c|c|c|c|c|}
  \hline
$ \delta t$ & $L^{1}$ Error  &Convergence Rate & $L^{\infty}$ Error    & Convergence Rate \\
\hline
  0.04&0.0676&--& 0.0928&--\\
  \hline
0.02 &0.0347& 0.95&0.0621&0.49\\
\hline
 0.01 &0.0170&1.04&0.0383&0.62\\
\hline
  0.005&0.0079&1.14& 0.0215&0.79\\

 \hline
\end{tabular}
\end{center}
}
\end{table}

}

\subsection{A time refinement scheme for contact point motion } \label{sec:modifiedalg}

For any given space mesh, the only parameter in Algorithm I is the time step $\delta t$. According to  Merriman, Bence and Osher \cite{merriman1992diffusion}, the choice of $\delta t$ should meet two requirements. The first one is that $\delta t$ should be small enough so that the approximation of the energy is reasonably accurate. The second is that $\delta t$ should also be large enough so that the boundary curve moves at least one grid cell on the spatial grid (otherwise the interface would not move after the thresholding step), that is, $\delta t \gg \frac{\delta x} {\kappa}$ where ${\kappa}$ is the average  curvature and $\delta x$ is the space mesh size. Since we have volume conservation, the interface will eventually become circular with a constant  curvature. Therefore, for a given space mesh size $\delta x$, there is a $\delta t$ threshold below which the interface will not move.  Therefore time step refinement beyond this threshold will   not improve the accuracy of the interface location.   However,  when the interface intersects  the solid boundary,  the motion of the contact point follows  different dynamics and is controlled by the Young stress $f=\gamma_{LV}(\cos\theta-\cos\theta_Y)$. This may lead  to a different time scale (and a different time step constraint).  Numerical results show that time step refinement improves the accuracy near the contact point.   Hence, we propose a time refinement  scheme to minimize the interfacial energy.   The idea is to first use a proper  (large enough) time step $\delta t$ so that the evolution of the interface reaches equilibrium. We then improve the contact point accuracy by  repeatedly halving the time step $\delta t$ until the difference between the solutions of succeeding steps is within a tolerance $\eps_1$.

\medskip

{ {\bf Modified Algorithm I}
\it
\begin{itemize}\item[]
\begin{description}
\item[Step 0.] Given initial $D_1^0, D_2^0\subset\Omega$, such that $D_1^0\cap D_2^0=\emptyset$, $D_1^0\cup D_2^0=\Omega$ and
$|D_1^0|=V_0$. Set $D_1^{*}=D_1^0$. Set a tolerance parameter $\eps>0$.
\item[Step 1.]For given set $(D_1^k, D_2^k)$, we define two functions
\begin{align}
\phi_1=\frac{1}{\sqrt{\delta t}} G_{\delta t}*(\gamma_{LV} \chi_{D_2^k}+\gamma_{SL}\chi_{D_3}), \
\phi_2=\frac{1}{\sqrt{\delta t}} G_{\delta t}*(\gamma_{LV} \chi_{D_1^k}+\gamma_{SV}\chi_{D_3}).
\end{align}
\item[Step 2.] Find a constant $\delta$ to ensure volume preservation using the quick-sort algorithm  in section~\ref{sec:volumeconservation}, so that the set
\begin{equation}
\tilde{D}_1^{\delta}=\{ x\in\Omega | \phi_1 < \phi_2+\delta. \}
\end{equation}
satisfying $|\tilde{D}_1^{\delta}|\approx V_0$.
Denote $D_1^{k+1}=\tilde{D}_1^{\delta}$, $D_2^{k+1}=\Omega\setminus D_1^{k+1}$  .
\item[Step 3.]{IF $|D_1^{k}-D_1^{k+1}|\leq \eps,$\newline
\hspace{1cm} if $ |D_1^{*}-D_1^{k+1}|\geq \eps, $set $ \delta t=\frac{\delta t}{2}$, $D_1^{*}=D_1^{k+1}$, and go back to step 1.\\
$ \quad \quad \quad$ else, set $D_1^{*}=D_1^{k+1}$ and  stop. \\
$ \quad \quad \quad$ endif \\
$\quad$  ELSE,  go back  to step 1. \\
$\quad$ ENDIF}
\end{description}
\end{itemize}
}

\subsection{Accuracy check of the Modified Algorithm I} \label{sec:accuracycheck}
To check the accuracy of the  Modified  Algorithm I described in Section~\ref{sec:modifiedalg}, we consider  a two-dimensional drop spreading on a solid surface. The equilibrium state is a circular arc with  Young's  angle  when  the minimum of the total interfacial energy is reached.
In our experiment, the initial liquid phase is given by a semi-circle centered at $(0,-\frac{\pi}{4})$ with radius $\frac{\pi}{4}$. So the volume of the drop is $\frac{\pi^3}{32}$. We set three surface tensions as  $\gamma_{LV}=1,\gamma_{LS}=1$ and $\gamma_{SV}=1+\sqrt{3}/2$, which gives  Young's  angle $\frac{\pi}{3}$. In this case,  the exact  equilibrium state can be computed
explicitly.

In Figure~\ref{fig:256adaptintime}, Figure~\ref{fig:512adaptintime} and Figure~\ref{fig:1024adaptintime}, we show the errors of solutions (characteristic functions) computed  by both Algorithm I and the Modified Algorithm I, compared with the exact solution (the characteristic function of the exact equilibrium state) which shows the location error of the interface.    It is obvious that the errors near contact points are much larger than those at other places of the interfaces. However, after time step refinement, the Modified Algorithm 1 gives much improved results. Figure~\ref{fig:exact_numerical} compares well  the numerical solution and exact solution at the equilibrium.

We then check the accuracy of the algorithms via calculating the convergence rate of the $L^1$ error
and $L^\infty$ error with respect to the mesh refinement.  Table~\ref{tab:convergencerate}  shows the $L^1$ errors of both schemes. Again the Modified Algorithm I gives much better results.  The  results also show that the convergence rate for $L^1$ error of our algorithm is of first order.
{  Table~\ref{tab:convergencerateinfty}  shows the $L^\infty$ errors of both schemes. Again the Modified Algorithm I gives superior  results. The  example shows that
the time refinement scheme improves the accuracy dramatically. But this does not necessarily  mean that the convergence order is also improved, especially for the $L^1$ error.}

\begin{figure}
\begin{center}
\includegraphics[scale=0.2]{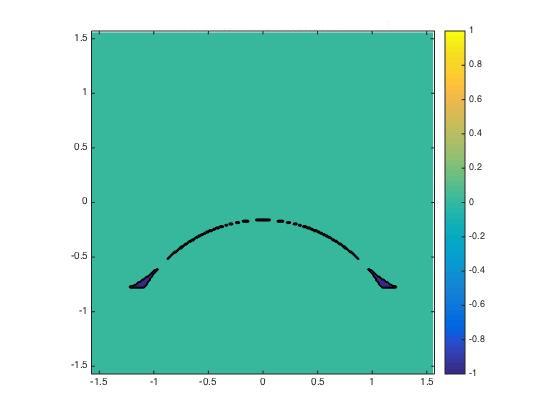}
\includegraphics[scale=0.2]{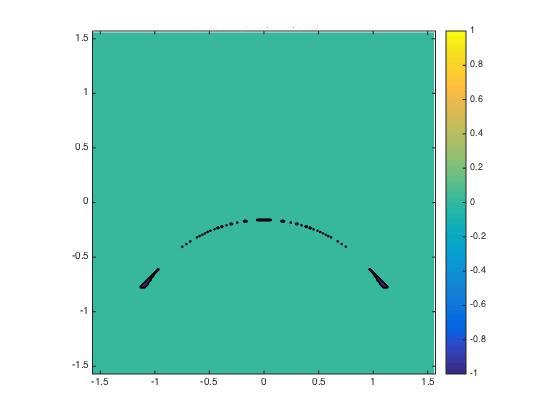}
\caption{ Left: 256 $\times$ 256 grid points, $\delta t=2dx$ without refinement in time. Right: 256 $\times$ 256 grid points, $\delta t=2dx$ initially with refinement in time, $\epsilon=1.0e^{-10}$} \label{fig:256adaptintime}
\includegraphics[scale=0.2]{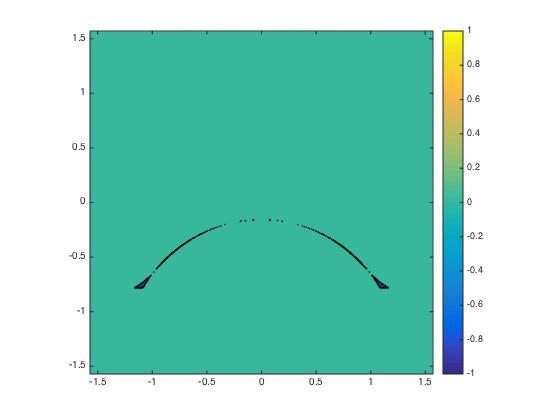}
\includegraphics[scale=0.2]{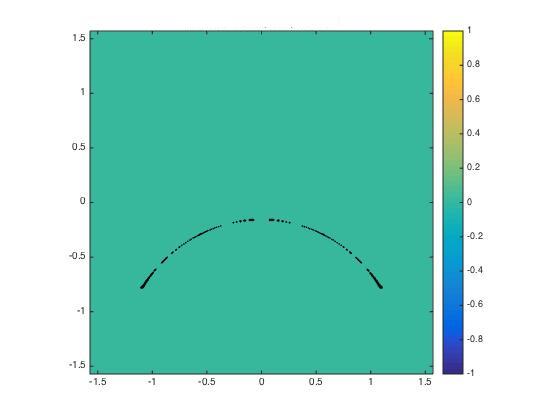}
\caption{Left: 512 $\times$ 512 grid points, $\delta t=2dx$ without refinement in time. Right: 512 $\times$ 512 grid points, $\delta t=2dx$ initially with refinement in time, $\epsilon=1.0e^{-10}$} \label{fig:512adaptintime}
\includegraphics[scale=0.2]{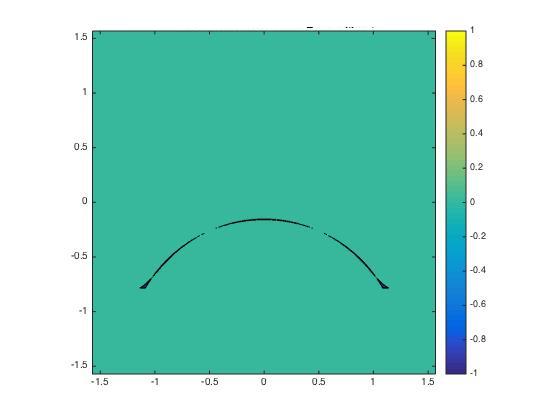}
\includegraphics[scale=0.2]{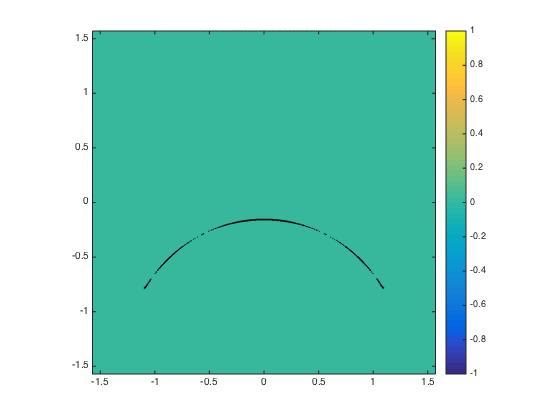}
\caption{Left: 1024 $\times$ 1024 grid points, $\delta t=2dx$ without refinement in time. Right: 1024 $\times$ 1024 grid points, $\delta t=2dx$ initially with refinement in time, $\epsilon=1.0e^{-10}$} \label{fig:1024adaptintime}
\end{center}
\end{figure}

\begin{figure}
\begin{center}
\includegraphics[scale=0.4]{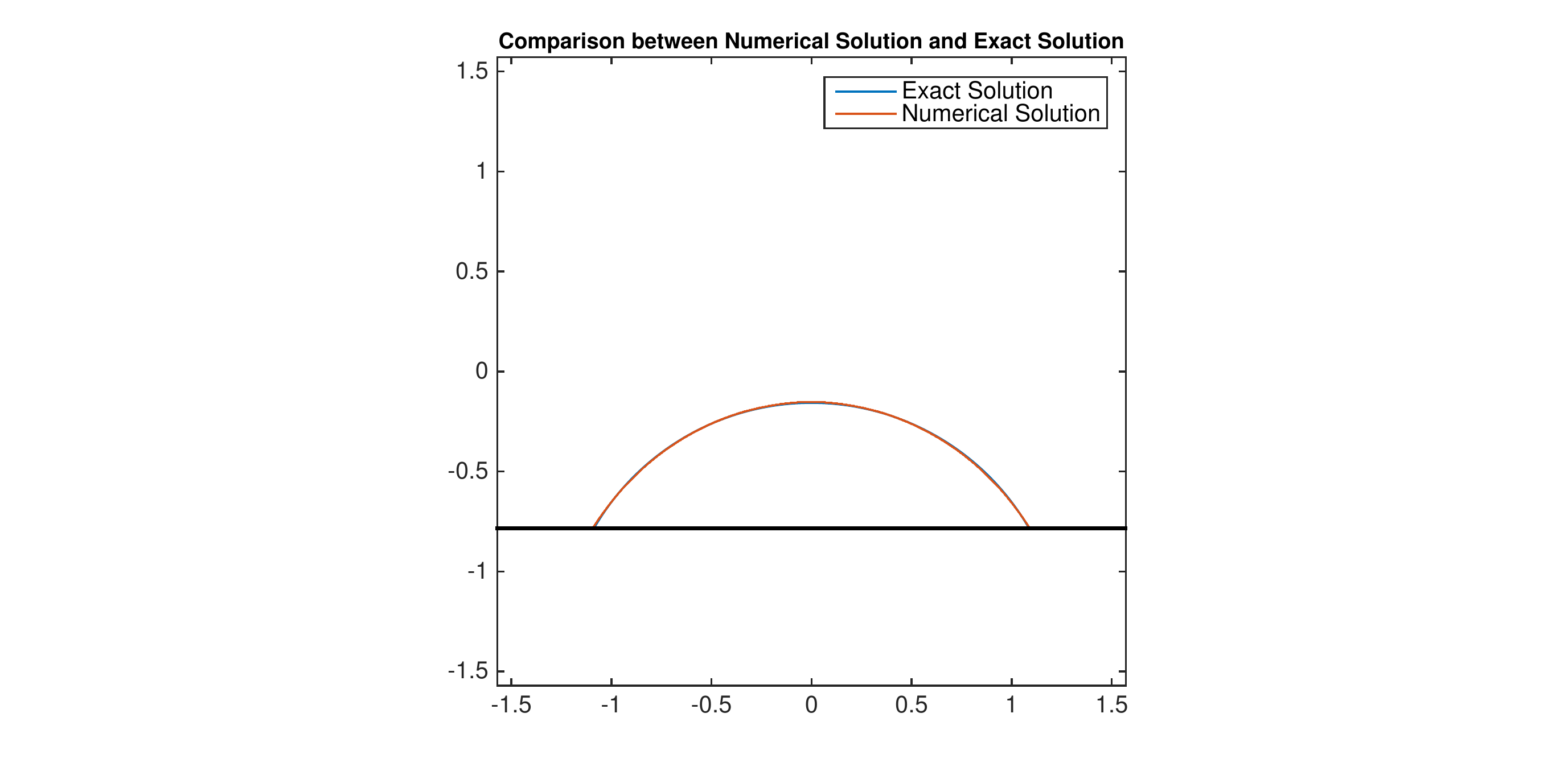}
\end{center}
\caption{Comparison of the numerical solution at equilibrium  to the exact solution.  Red line represents the exact solution while the blue line represents the numerical solution (computed with 1024 $\times$ 1024 grid points).}  \label{fig:exact_numerical}
\end{figure}

\begin{table}
\caption{Accuracy Check  in  $L^1$ norm}\label{tab:convergencerate}
\begin{center}
\begin{tabular}{|c|c|c|c|c|}
  \hline
 Grid points & $L^{1}$ error  & Convergence & $L^{1}$ error    & Convergence \\
 &&rate &with  time refinement& rate\\
\hline
128$ \times$ 128 &0.1473&-&0.0515&--\\
\hline
 256$ \times$ 256 &0.0482&2.06&0.0271&0.90\\
\hline
  512$\times$ 512 &0.0200&1.41& 0.0109&1.49\\
\hline
  1024 $\times$ 1024&0.0116&0.72& 0.0054&1.02\\
 \hline
\end{tabular}
\end{center}
\end{table}

\begin{table}
{
\caption{Accuracy Check in  $L^{\infty}$  norm}\label{tab:convergencerateinfty}
\begin{center}
\begin{tabular}{|c|c|c|c|c|}
  \hline
 Grid points & $L^{\infty}$ error  &Convergence & $L^{\infty}$ error    & Convergence \\
 &&rate &with  time refinement& rate\\
\hline
128$ \times$ 128 &0.1473&-&0.0982&--\\
\hline
 256$ \times$ 256 &0.0831&0.77&0.0585&0.68\\
\hline
  512$\times$ 512 &0.0552&0.51& 0.0307&0.91\\
\hline
  1024 $\times$ 1024&0.0333&0.66& 0.0149&1.1\\
 \hline
\end{tabular}
\end{center}
}
\end{table}

 \section{A drop spreading on a chemically pattern solid surface} \label{sec:chemicallypattern}
We first study the hysteresis behavior of a drop spreading on a chemically patterned surface.  We consider the  quasi-static spreading of a drop. To simulate the hysteresis process.
we need to increase or decrease the volume of the drop gradually.  In each step, we need to compute the equilibrium state of the drop after liquid is added or extracted, which is very computationally demanding.  We show that our threshold dynamics method
can simulate the process efficiently.

We assume that the surface is periodically patterned in the  interval $(-\pi/2,\pi/2)$ and the interval is divided into $2k+1$ periods with an equal partition of two materials $\mathcal{A}, \mathcal{B}$ away from the center. The center  part is occupied by the material $\mathcal{B}$ (See  Figure~\ref{fig:chemicallypatternsolid}). Assume $\theta_{\mathcal{A}}$, $\theta_{\mathcal{B}}$ are  Young's angles for  materials $\mathcal{A}$ and $\mathcal{B}$  respectively. $r$ is the initial radius of a semi-circle on the surface and $\Delta V$ is the volume we add to the drop each time.
The procedure for explicitly calculating the change in contact angle and position of contact points with respect to the volume for  simple two-phase systems on a chemically patterned surface is given in \cite{XuWang2011}.

\begin{figure}
\begin{center}
\includegraphics[scale=0.5]{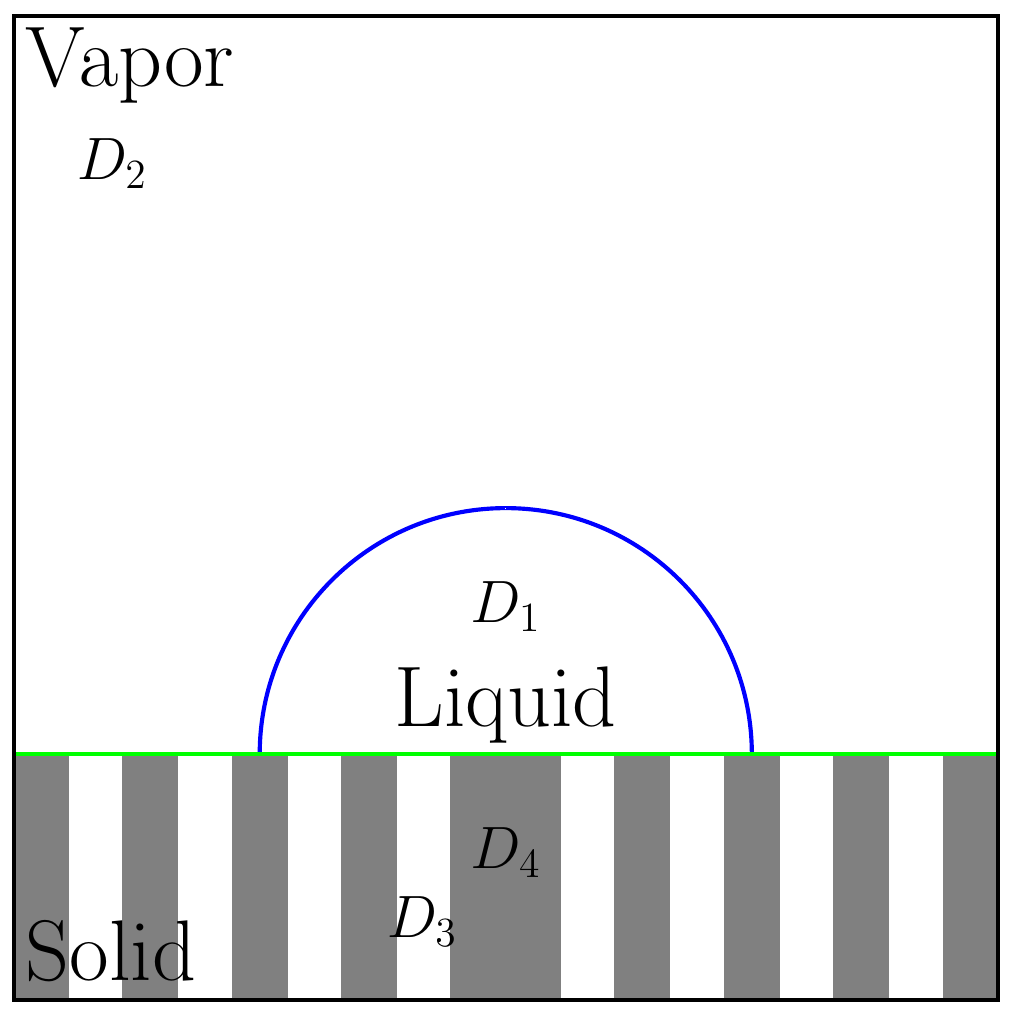}
\end{center}
\caption{A sketch of a drop spreading on a  chemically patterned solid surface. Here $D_3$ (white region) and $D_4$ (shaded region)  represent materials $\mathcal{A}$ and $\mathcal{B}$ respectively. }\label{fig:chemicallypatternsolid}
\end{figure}
\begin{figure}
\begin{center}
\includegraphics[scale=0.23]{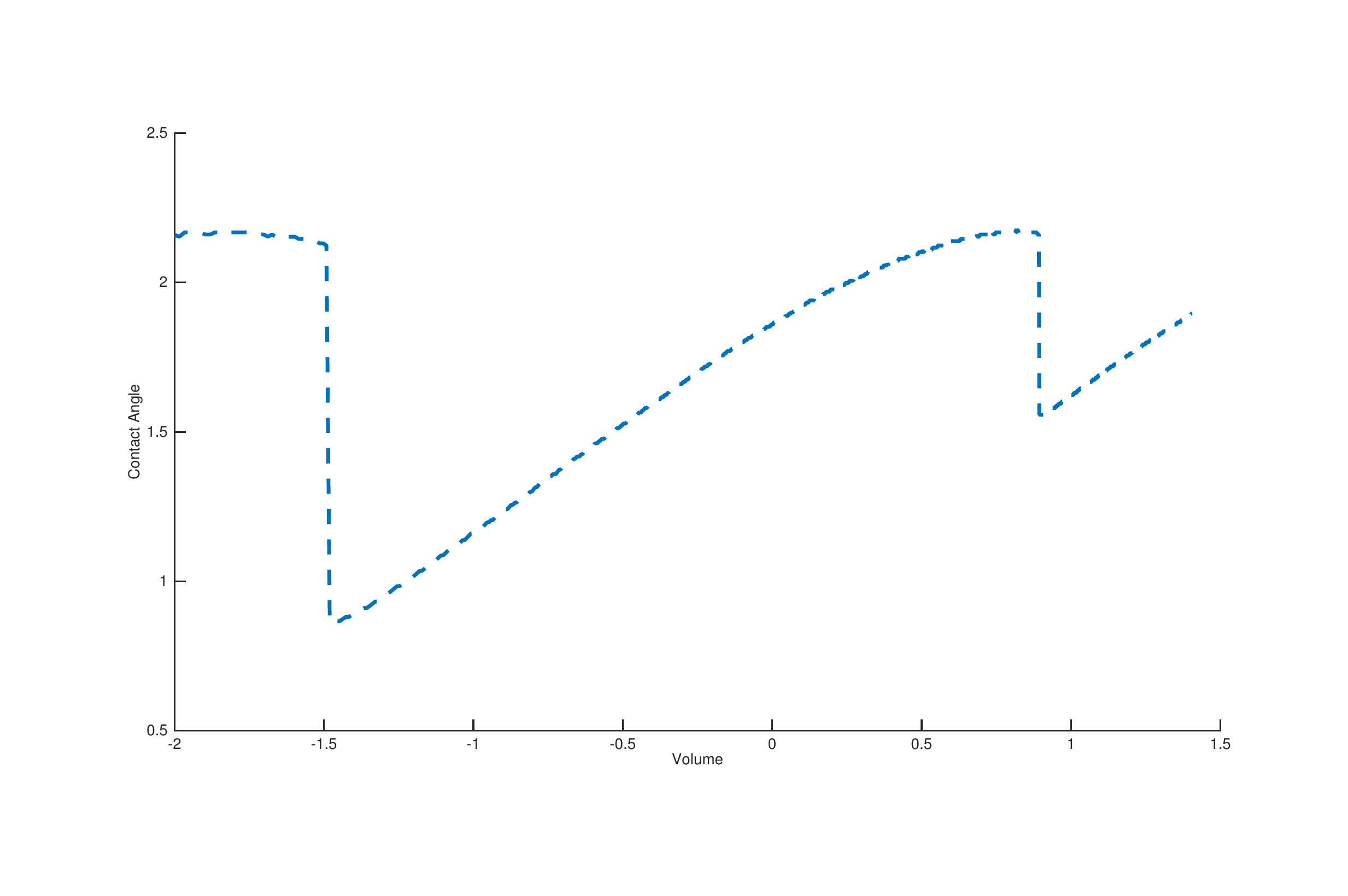}   \includegraphics[scale=0.23]{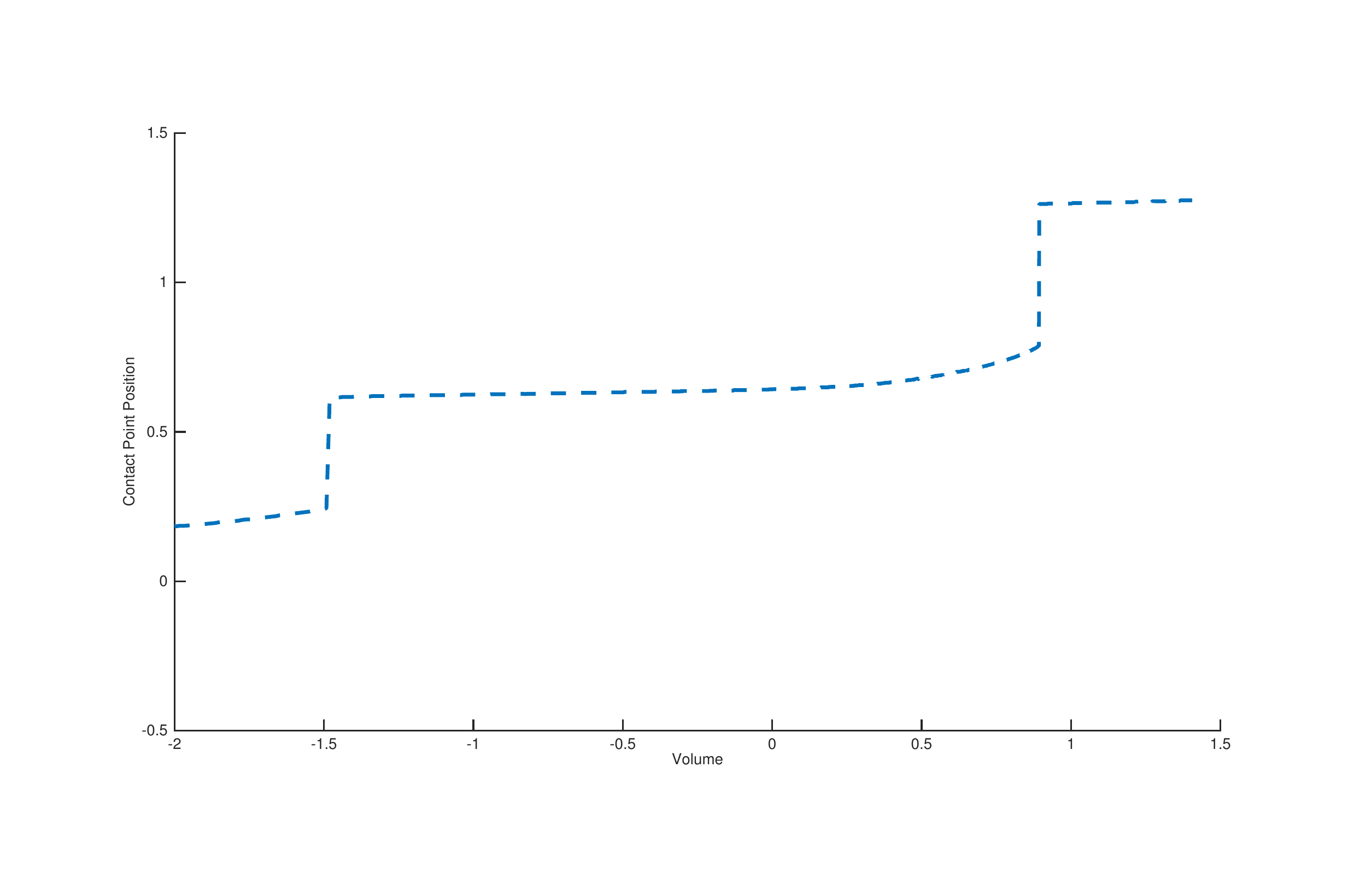}
\end{center}
\caption{The  stick-slip motion of a drop with k=2 when the volume is increasing. $\theta_{\mathcal{A}}=\frac{\pi}{5}$, $\theta_{\mathcal{B}}=\frac{7\pi}{10}$.  }\label{fig:CAH_K_2}
\end{figure}
\begin{figure}
\begin{center}
\includegraphics[scale=0.23]{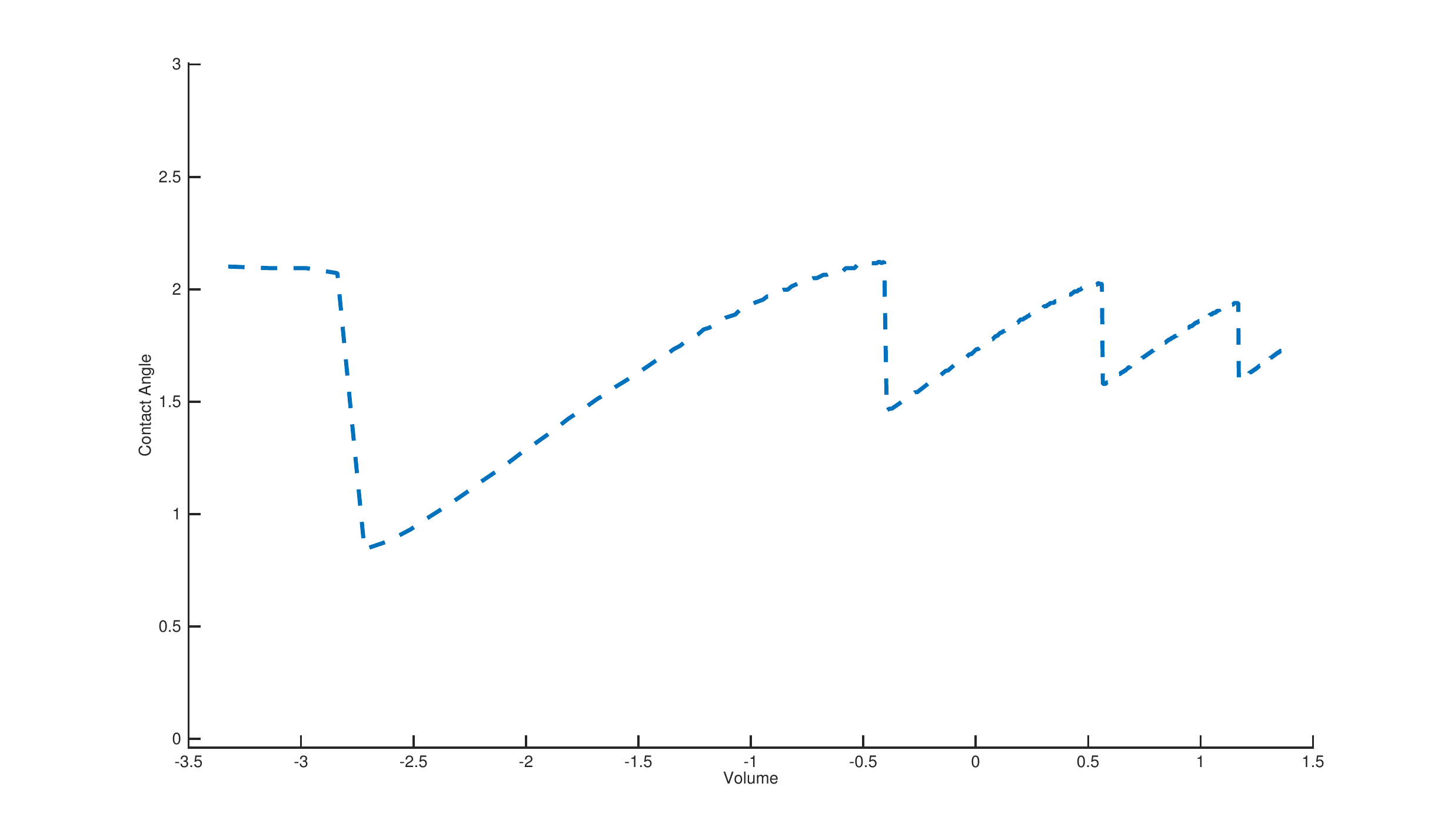}   \includegraphics[scale=0.23]{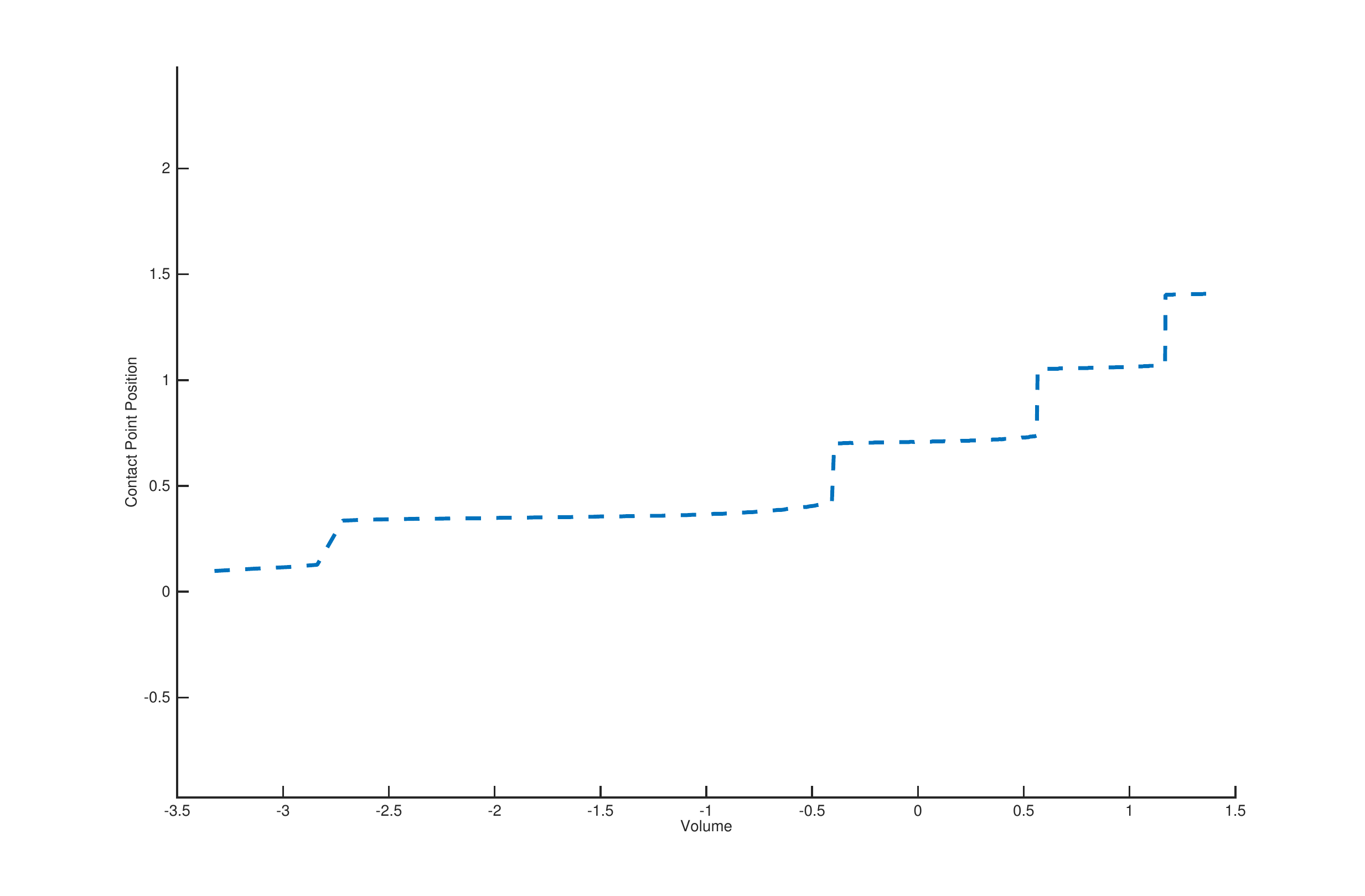}
\end{center}
\caption{The  stick-slip motion of a drop with k=4 when the volume is increasing. $\theta_{\mathcal{A}}=\frac{\pi}{5}$, $\theta_{\mathcal{B}}=\frac{7\pi}{10}$.  }\label{fig:CAH_K_4}
\end{figure}
\begin{figure}
\begin{center}
\includegraphics[scale=0.23]{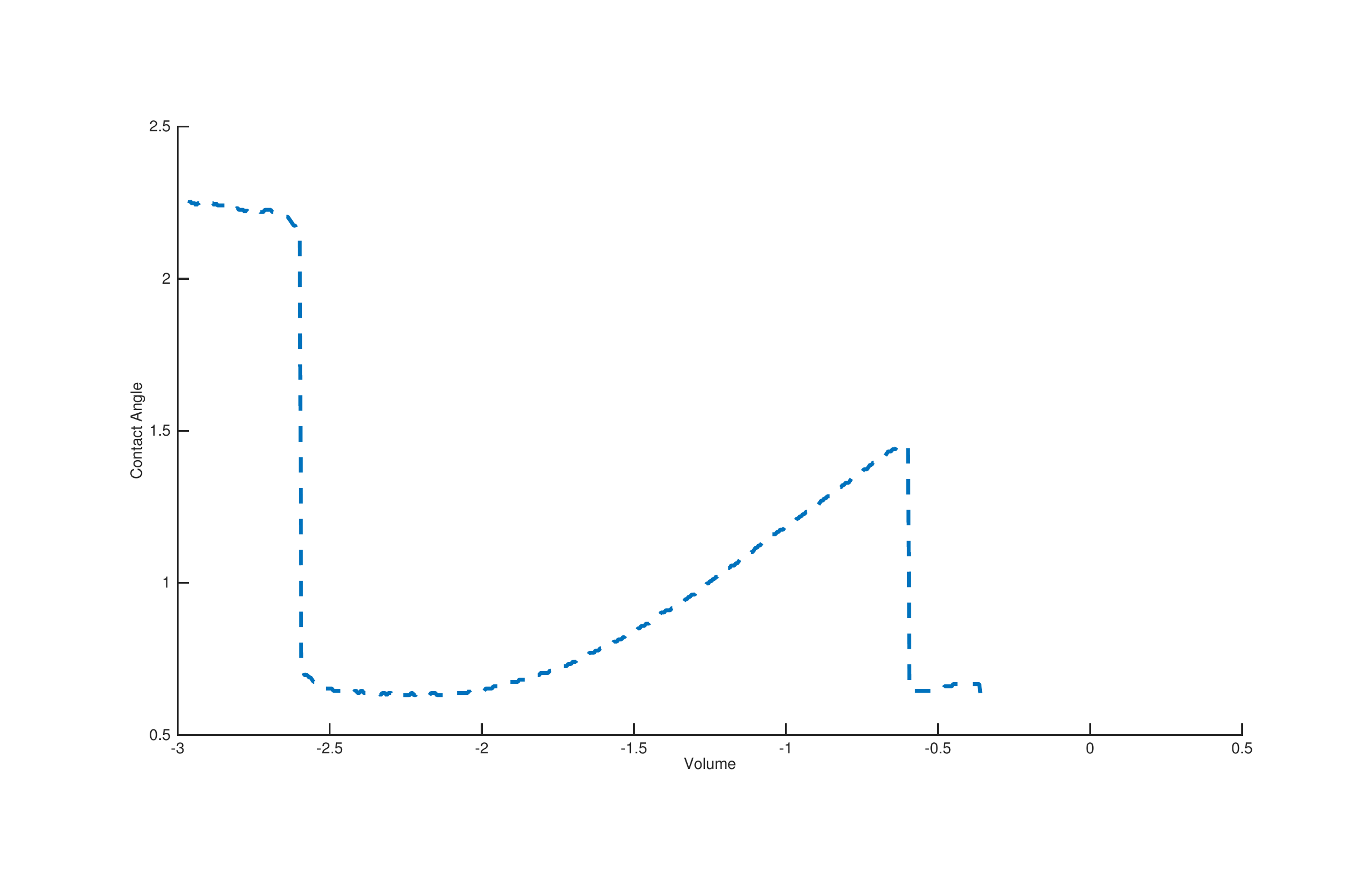}   \includegraphics[scale=0.23]{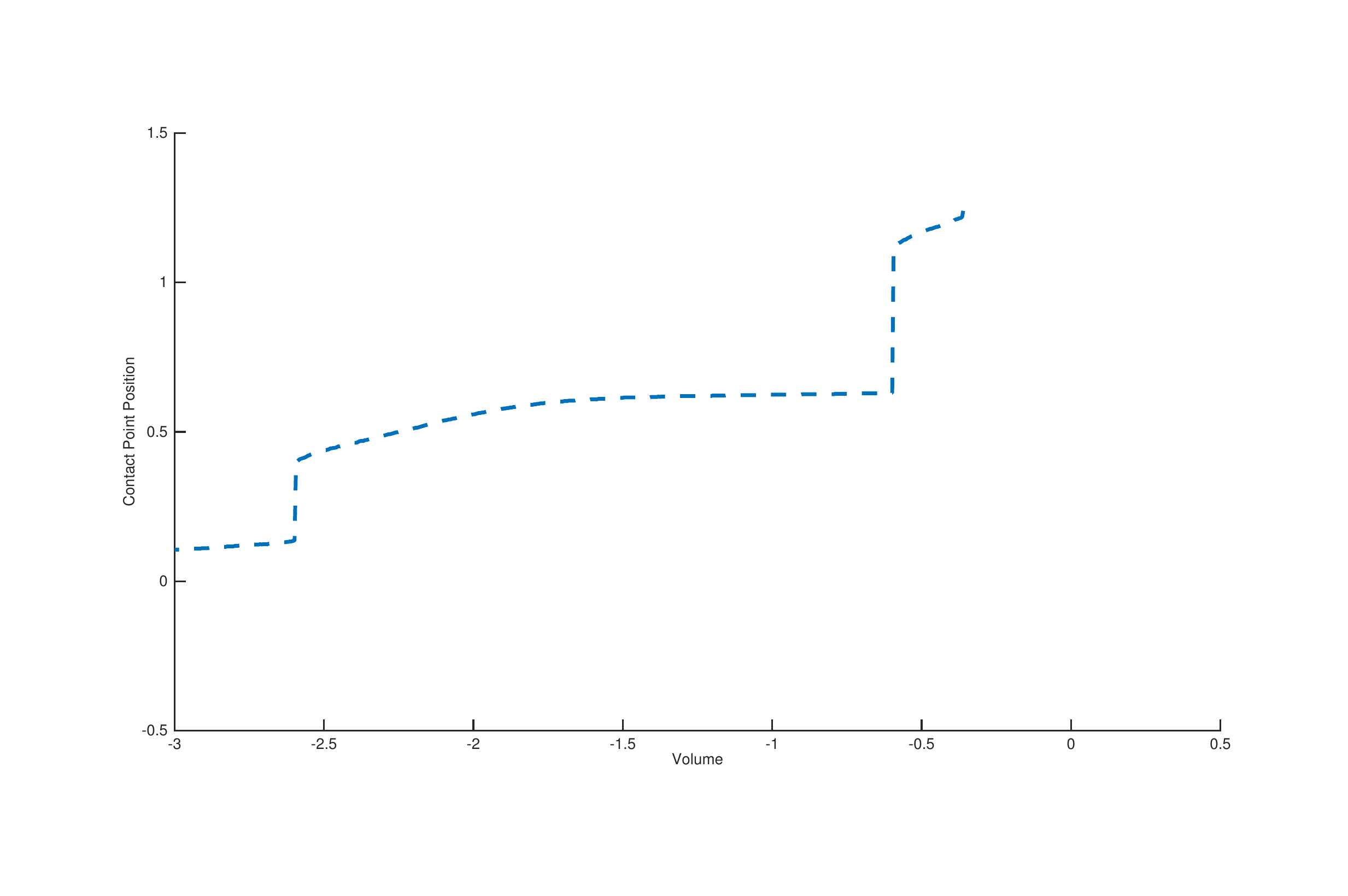}
\end{center}
\caption{The  stick-slip motion of a drop with k=2 when the volume is decreasing. $\theta_{\mathcal{A}}=\frac{\pi}{5}$, $\theta_{\mathcal{B}}=\frac{7\pi}{10}$.  } \label{fig:CAH_K_2_DE}
\end{figure}
\begin{figure}
\begin{center}
\includegraphics[scale=0.23]{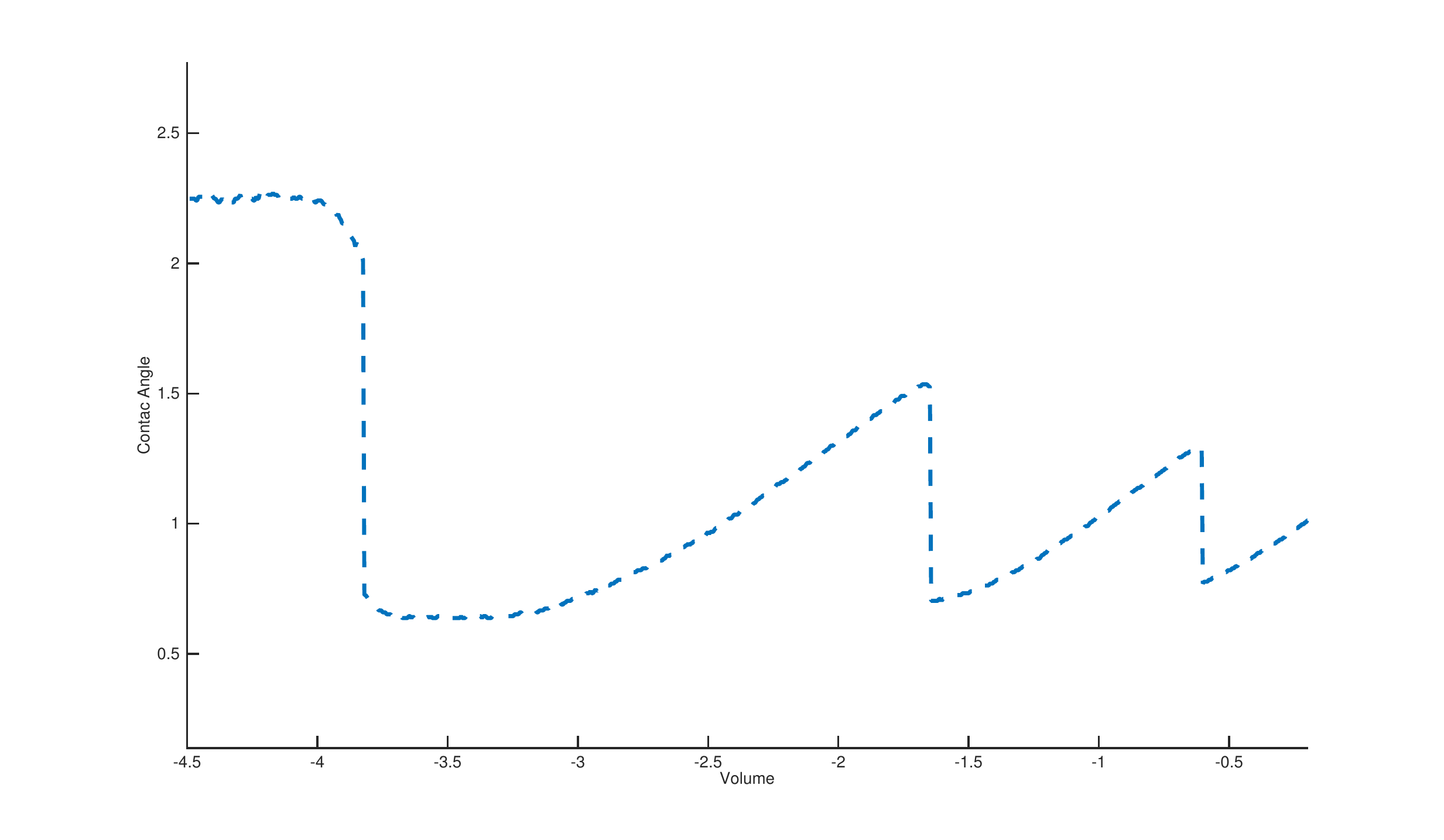}   \includegraphics[scale=0.23]{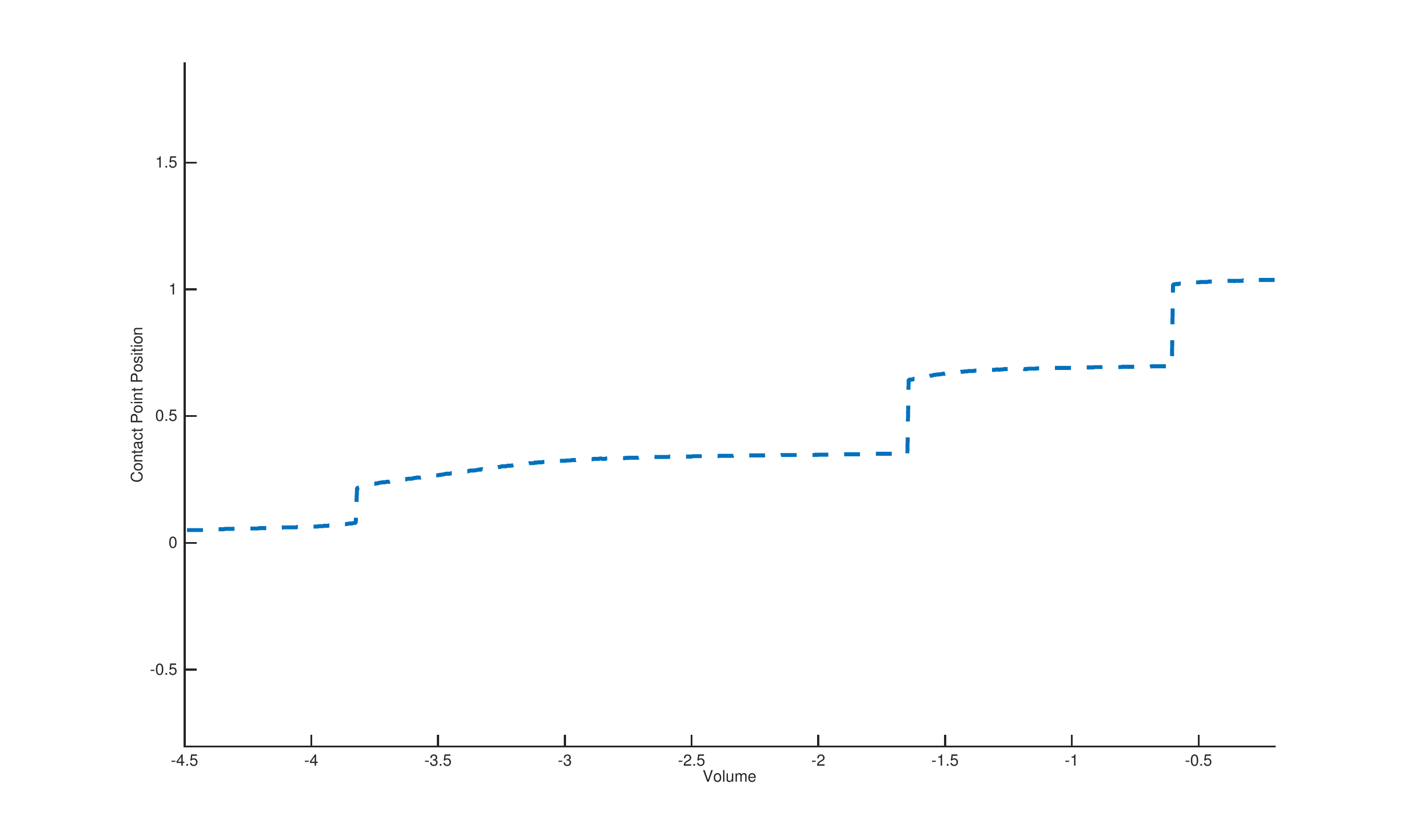}
\end{center}
\caption{The stick-slip motion of a drop with k=4 when the volume is decreasing. $\theta_{\mathcal{A}}=\frac{\pi}{5}$, $\theta_{\mathcal{B}}=\frac{7\pi}{10}$.  }\label{fig:CAH_K_4_DE}
\end{figure}

To implement the Modified Algorithem I,  we need to divide our solid region into two parts $D_3$ and $D_4$ representing  material $\mathcal{A}$ and  material $\mathcal{B}$ with different surface tensions, respectively (as shown in Figure~\ref{fig:chemicallypatternsolid}),  and  modify the original $\gamma_{SL}\chi_{D_3}$ and $\gamma_{SV}\chi_{D_3}$ to $\gamma_{S_1L}\chi_{D_3}+\gamma_{S_2L}\chi_{D_4}$ and $ \gamma_{S_1V}\chi_{D_3}+\gamma_{S_2V}\chi_{D_4}$.  As the volume of the drop increases quasi-statically, we use the Modified Algorithem I to calculate the equilibrium state for each fixed volume.

 We take $\theta_{\mathcal{A}}=\frac{\pi}{5}$, $\theta_{\mathcal{B}}=\frac{7\pi}{10}$.  For the advancing drop,  we plot the contact angle and position of contact point as functions of increasing volume  in Figure~\ref{fig:CAH_K_2} for $k=2$ and  in Figure~\ref{fig:CAH_K_4} for $k=4$.    The contact point goes through the stick-slip motion, and the contact angle oscillates near the advancing angle  $\theta_{\mathcal{B}}$ for larger $k$.

For the receding drop, we plot the contact angle and location of the contact point as functions of increasing volume  in Figure~\ref{fig:CAH_K_2_DE} for $k=2$ and  in Figure~\ref{fig:CAH_K_4_DE} for $k=4$.  Again, the contact point goes through the stick-slip motion,  and the contact angle oscillates near the receding angle  $\theta_{\mathcal{A}}$ for larger $k$.

{In  Figure~\ref{fig:pattern},  we show two quasi-static drops.   One is in the process of increasing in volume (advancing) and
the other is in the process of decreasing in volume (receding). We  see that the two states have very
different contact angles although the volume is the same.
This clearly shows that the contact angle hysteresis  as the shape of a drop on a chemically patterned surface depends on its
history.}
\begin{figure}
\begin{center}
\includegraphics[scale=0.35]{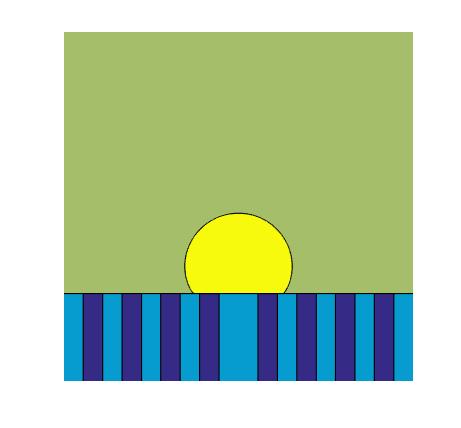}  \includegraphics[scale=0.35]{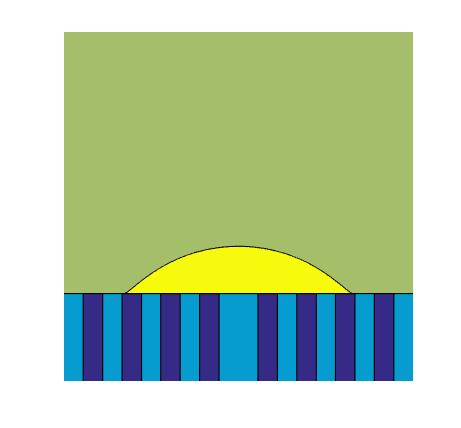}
\end{center}
\caption{Left: A quasi-static drop in the process of growing in  volume on a chemically patterned surface when the initial  volume is 0.5883. Right: A quasi-static drop in the process of shrinking in volume on a chemically patterned surface when the initial volume is 0.5883.
Young's angle in the light blue part is $\mathcal{B}$ while that in the dark blue part is $\mathcal{A}$. Here k=4. }\label{fig:pattern}
\end{figure}

\section{A drop spreading on a rough solid surface}  \label{sec:rough}
\begin{figure}
\vspace{-0.1cm}
\begin{center}
\includegraphics[scale=0.5]{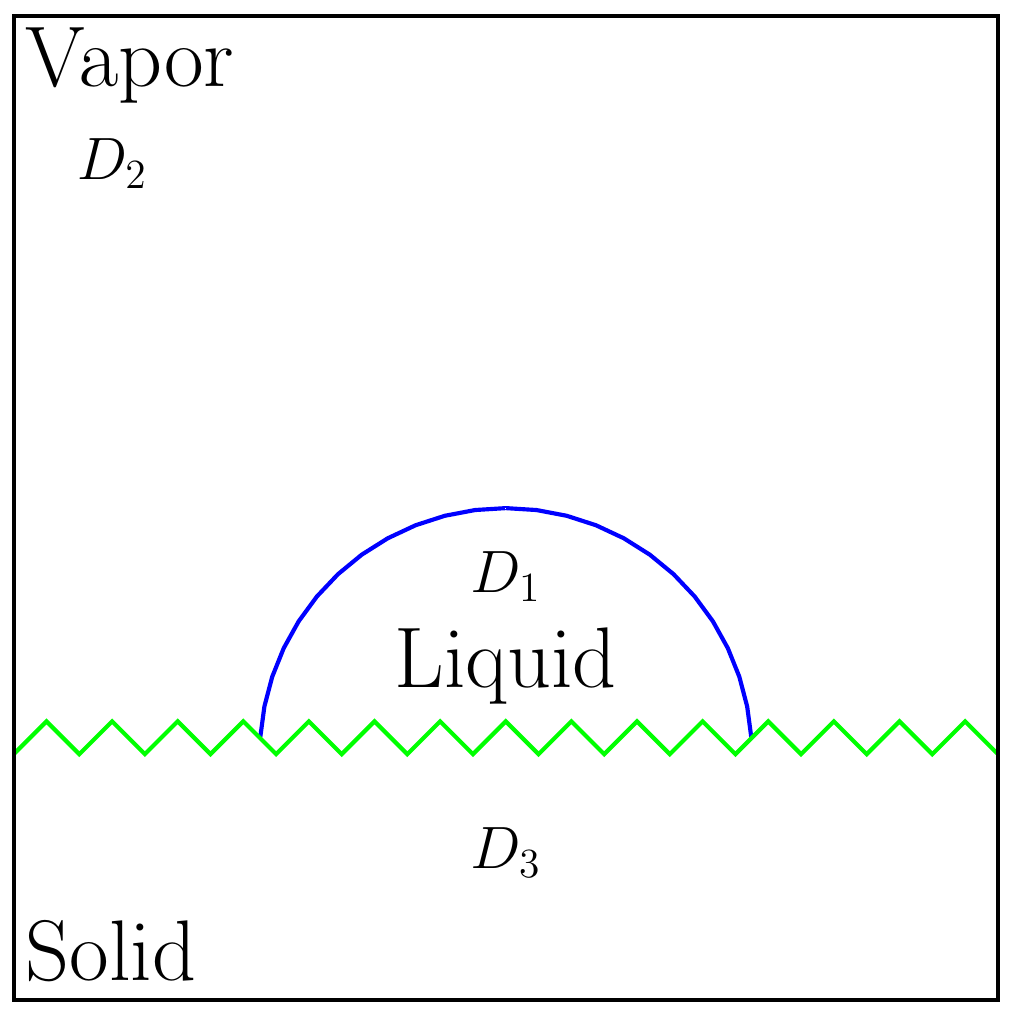}  
\end{center}
\caption{Left:  A sketch of a drop spreading on a rough solid surface.  The solid surface is given by  a sawtooth profile. }
\label{fig:roughsolid}
\begin{center}
 \includegraphics[scale=0.23]{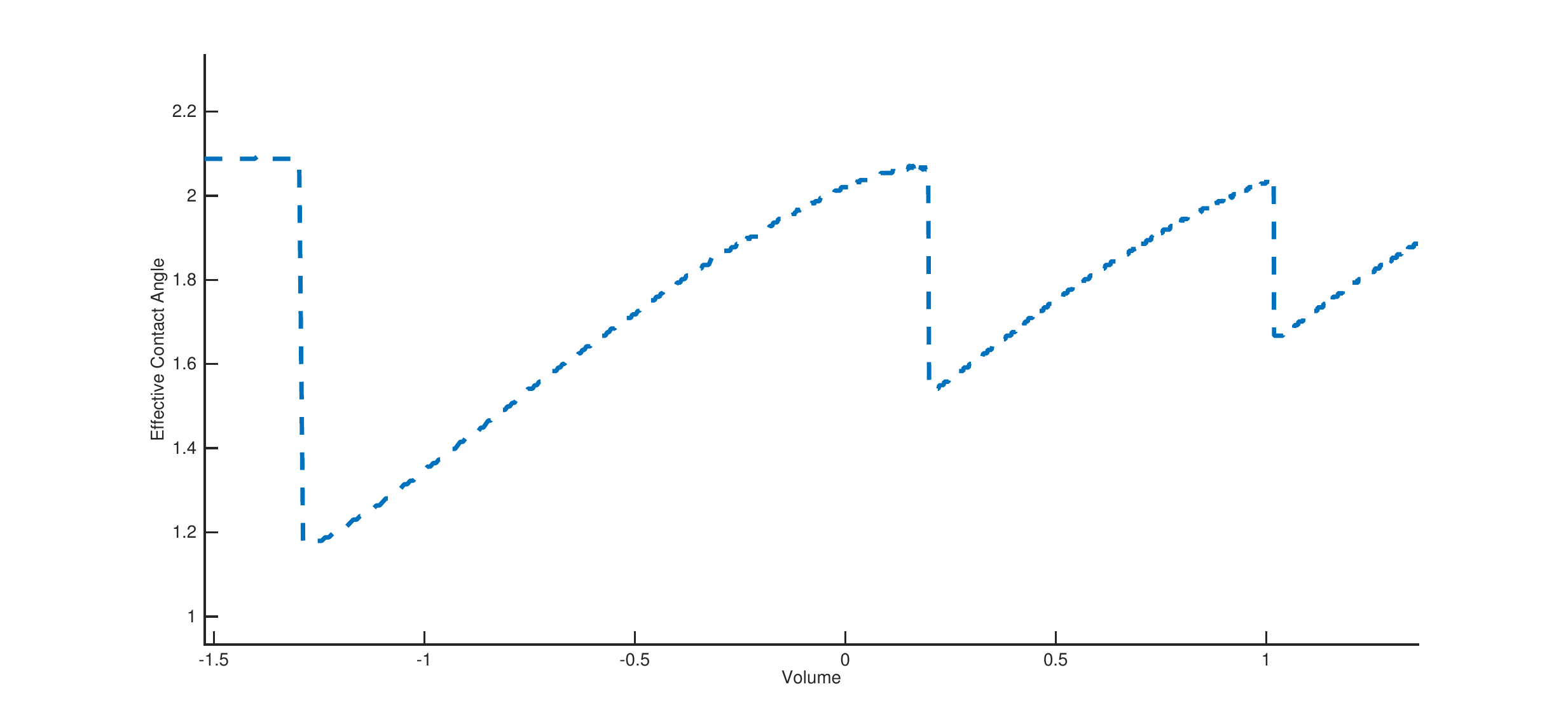}   \includegraphics[scale=0.23]{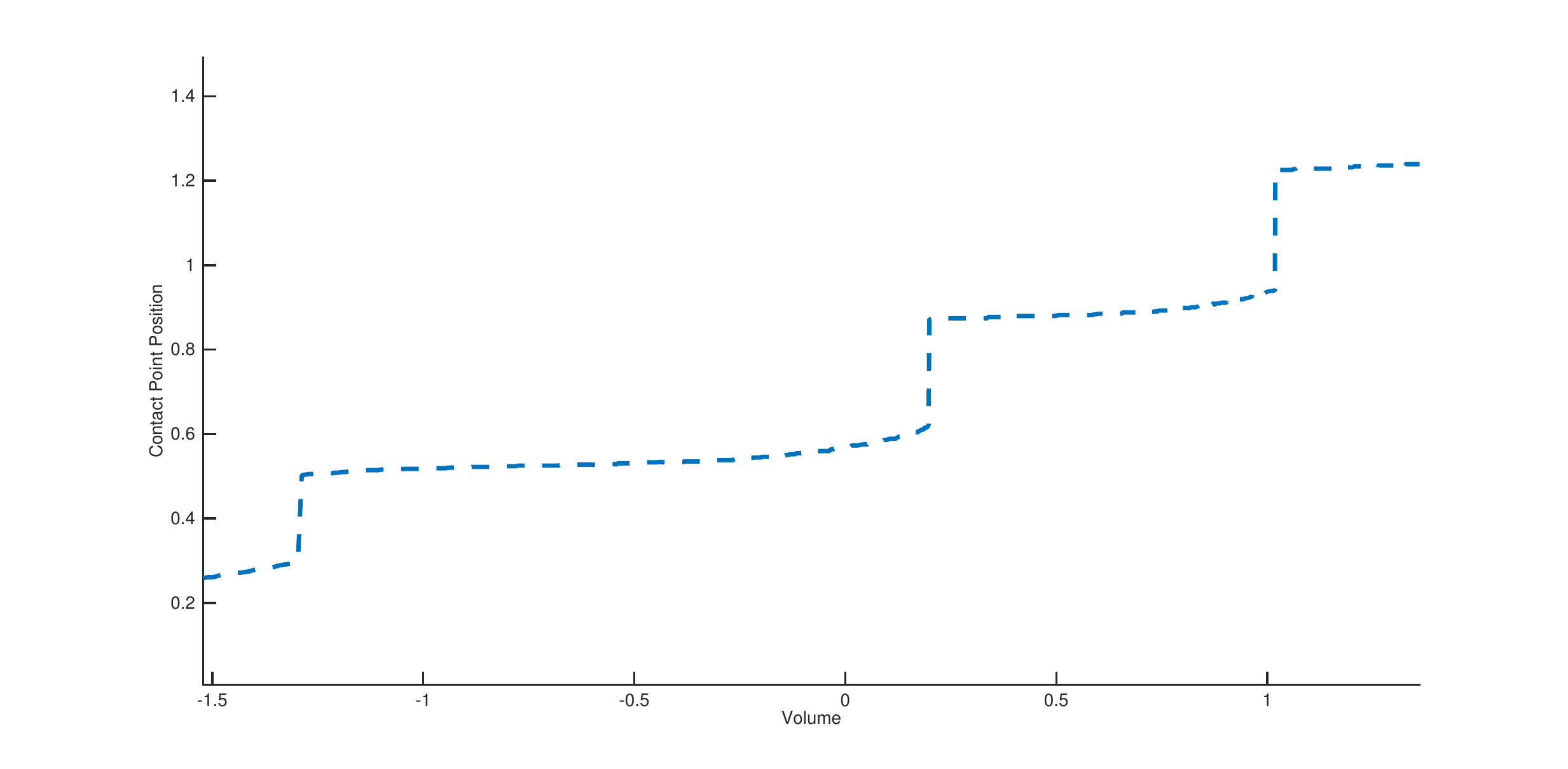}
\end{center}
\caption{The  stick-slip motion of a drop on a rough surface when the volume is increasing.  $\theta=\frac{\pi}{2}, k=4, \alpha=\frac{\pi}{6}$. }\label{fig:ROUGH}
\begin{center}
 \includegraphics[scale=0.23]{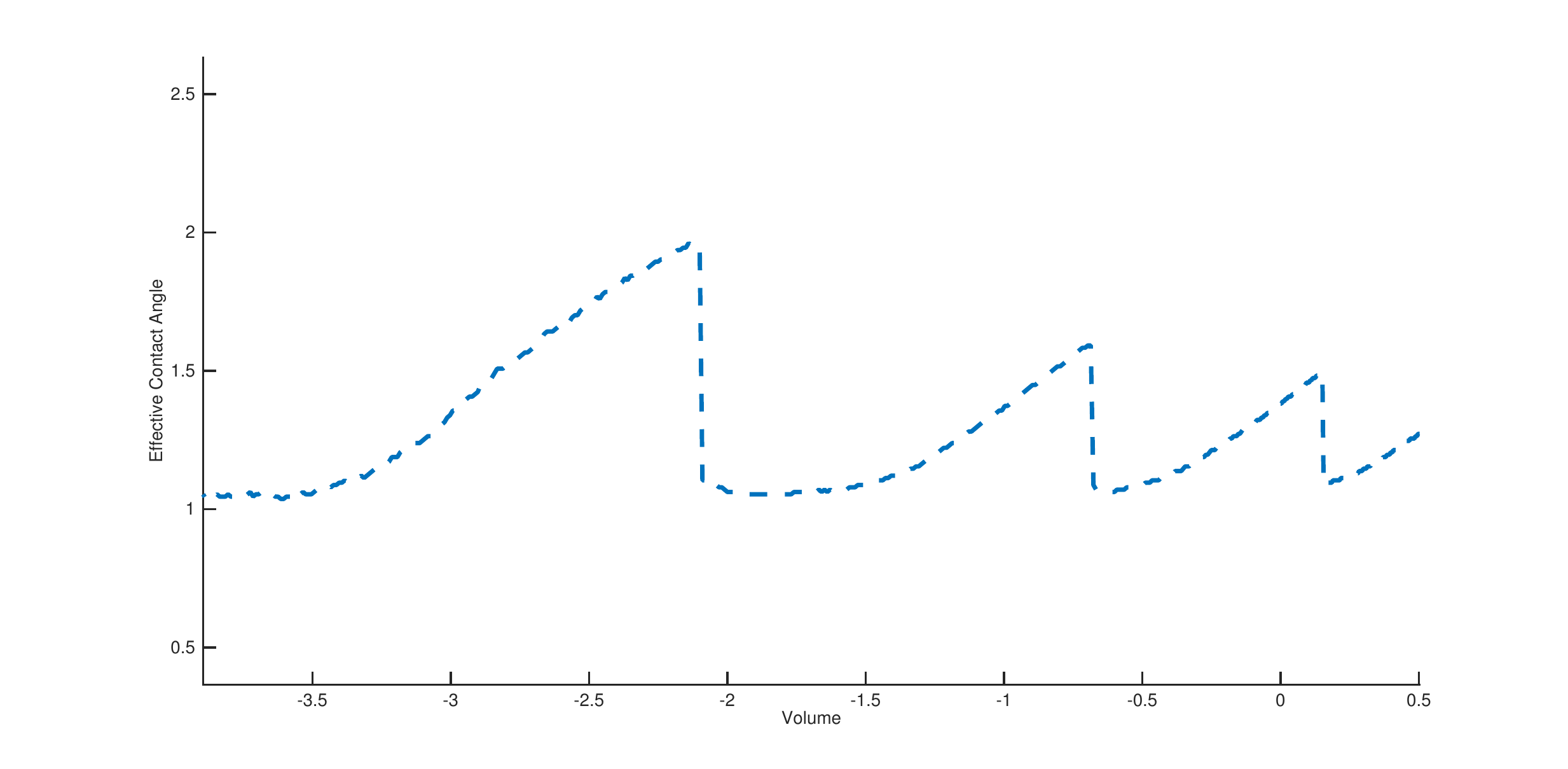}   \includegraphics[scale=0.23]{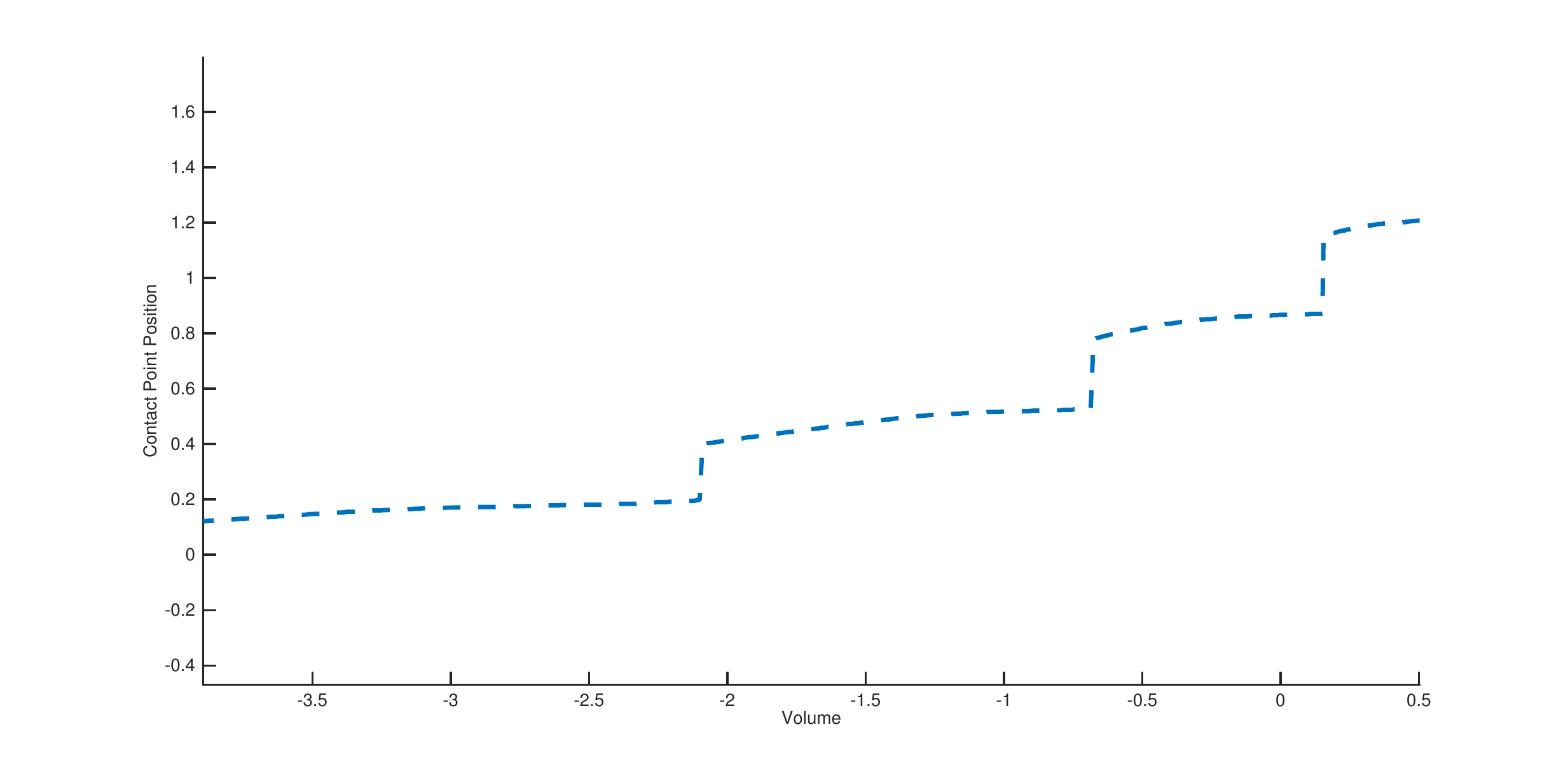}
\end{center}
\caption{The  stick-slip motion of a drop on a rough surface when the volume is decreasing.  $\theta=\frac{\pi}{2}, k=4, \alpha=\frac{\pi}{6}$. }\label{fig:ROUGH_DE}
\end{figure}
In this section, we will simulate the contact angle hysteresis on a geometrically rough surface.
In our experiments, the computational domain is $[-\frac{\pi}{2},\frac{\pi}{2}] \times [-\frac{\pi}{2},\frac{\pi}{2}]$, and we take the solid surface of shape given by a sawtooth function $$y=-\frac{\pi}{4}+\tan(\alpha)\frac{\pi}{4k+2}|s((2k+1)x-\pi)|$$ where $s(x)$ is a sawtooth periodic function with period $2\pi$ defined as
$$s(x)=\left\lbrace\begin{matrix}
\frac{2}{\pi}(x+\pi)-1  & -\pi \leq x \leq0;\\
-\frac{2}{\pi}x+1  &  0\leq x \leq \pi.
\end{matrix}\right.$$
For a rough surface, it is more meaningful to see how the effective contact angle behaves when the volume of the drop is increased or decreased \cite{chen2013effective}. The effective contact angle is defined as the angle between the contact line and the horizontal surface (See Figure~\ref{fig:roughsolid}). Figure~\ref{fig:ROUGH} and Figure~\ref{fig:ROUGH_DE} show
the bahavior of the contact angle and the $x$-coordinate of the contact point for the case when $k=4, \alpha=\frac{\pi}{6}$.  Young's angle of the solid surface is $\theta_Y=\frac{\pi}{2}$. We can see obvious stick-slip motion when we increase or decrease the volume of the drop.
Furthermore, the advancing contact angle is almost $\frac{2\pi}{3}$ and the receding contact angle is approximately
$\frac{\pi}{3}$.

{ In  Figure~\ref{fig:rough}, again, we show two quasi-static drops.   One is in the process of expanding in volume (advancing) and
the other is in the process of reducing in volume (receding).  Similar to the chemically patterned surface case,  the two states have very different apparent contact angles corresponding to the contact angle hysteresis  on rough surfaces.}

\begin{figure}
\begin{center}
\includegraphics[scale=0.32]{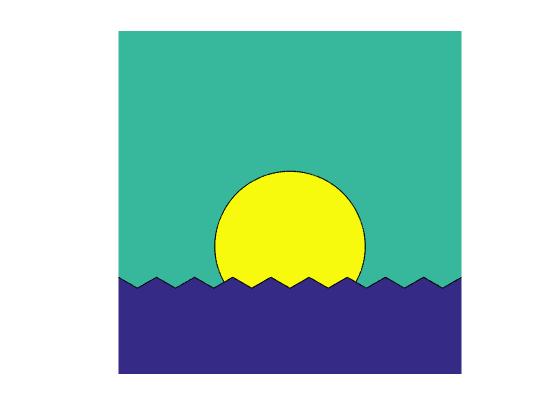}  \includegraphics[scale=0.32]{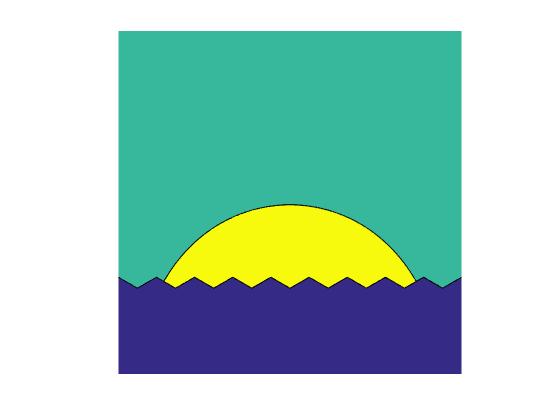}
\end{center}
\caption{Left: A quasi-static drop in the process of expanding in volume on a sawtooth rough surface when the volume is 1.178. Right: A quasi-static drop in the process of reducing in volume on a sawtooth rough surface when the volume is 1.178. $\theta=\frac{\pi}{2}, k=4, \alpha=\frac{\pi}{6}$.}
\label{fig:rough}\end{figure}

\section{Conclusion}
We develop an efficient threshold dynamics method for wetting on rough surfaces.  The method is based on minimization of the weighted surface area functional over an extended domain that includes the solid phase.   The method is  simple, stable with the complexity $O(N \log N)$ per time step and is not sensitive to the inhomogeneity or roughness  of the solid boundary.   The efficiency of the method can be further improved with adaptive mesh techniques with  more mesh points near the interface and contact line.

\section*{Acknowledgments}  This publication was based on work  supported in part by the Hong Kong RGC-GRF grants  605513 and 16302715 and the  RGC-CRF grant C6004-14G.  XM Xu is supported in part by  NSFC 11571354. We would also like to thank the anonymous referees for their
valuable comments and suggestions to help us to improve the paper.
\vspace{-0.1cm}




\section*{Reference}
\bibliographystyle{abbrv}
\bibliography{literatur.bib}

\end{document}